\documentclass[11pt,reqno]{amsart}
\usepackage{amsmath, amsfonts, amsthm, amssymb, mathtools, bm, longtable, booktabs,enumitem, mathrsfs, diagbox}
\usepackage[draft]{hyperref}
\usepackage{cite}

\usepackage{geometry}
\usepackage{verbatim}

\theoremstyle{plain}
\newtheorem{theorem}{Theorem}[section]
\newtheorem{corollary}[theorem]{Corollary}

\newtheorem{lemma}[theorem]{Lemma}
\newtheorem{proposition}[theorem]{Proposition}
\theoremstyle{definition}

\theoremstyle{remark}
\newtheorem*{remark}{Remark}
\newtheorem*{remarks}{Remarks}

\numberwithin{equation}{section}

\newcommand{\R}{\mathbb R}
\newcommand{\N}{\mathbb N}

\newcommand{\C}{\mathbb C}

\newcommand{\RE}{\operatorname{Re}}
\newcommand{\IM}{\operatorname{Im}}

\newcommand{\Log}{\operatorname{Log}}
\newcommand{\Arg}{\operatorname{Arg}}
\newcommand{\NN}{\mathcal{N}_N}

\renewcommand{\b}[1]{\boldsymbol{#1}}

\begin{document}
\allowdisplaybreaks


\begin{abstract}
We discuss two theorems in analytic number theory and combinatory analysis that have seen increased use in recent years. A corollary to a Tauberian theorem of Ingham allows one to quickly prove asymptotic formulas for arithmetic sequences, so long as the corresponding generating function exhibits exponential growth of a certain form near its radius of convergence. Two common methods for proving the required analytic behavior are
modular transformations and Euler-Maclaurin summation.
 However, these results are sometimes stated without certain technical conditions that are necessary for the complex analytic techniques that appear in Ingham's proof. We carefully examine the precise statements and proofs of these results, and find that in practice,
the technical conditions are satisfied for those cases appearing in recent applications. We also generalize the classical approach of Euler-Maclaurin summation in order to prove asymptotic expansions for series with complex values, simple poles, or multi-dimensional summation indices.
\end{abstract}

\title{On a Tauberian theorem of Ingham and Euler-Maclaurin summation}
\author[K. Bringmann]{Kathrin Bringmann}
\address{University of Cologne, Faculty of Mathematical and Natural Sciences, Mathematical Institute, Weyertal 86-90, 50931 Cologne, Germany}
\email{kbringma@math.uni-koeln.de}

\author[C. Jennings-Shaffer]{Chris  Jennings-Shaffer}
\address{Department of Mathematics, University of Denver, Denver, CO 80208, USA
\newline University of Cologne, Faculty of Mathematical and Natural Sciences, Mathematical Institute, Weyertal 86-90, 50931 Cologne, Germany}
\email{chrisjenningsshaffer@gmail.com}

\author[K. Mahlburg]{Karl Mahlburg}
\address{Department of Mathematics, Louisiana State University, Baton Rouge, LA 70803, USA}
\email{mahlburg@math.lsu.edu}

\thispagestyle{empty} \vspace{.5cm}
\maketitle

\section{Introduction and statement of results}

In mathematics one often encounters sequences $\{b_n\}_{n\in\N_0}$ whose terms enumerate the objects in some family of interest. Although the problem of finding closed-form expressions for the $b_n$ is often intractable, for some applications it is sufficient to determine the asymptotic behavior of $b_n$. A powerful technique is to consider the generating function of the sequence as a complex analytic power series, as its asymptotic analytic behavior can provide information about the asymptotic behavior of the $b_n$. In this article we revisit Ingham's Tauberian theorem \cite{Ingham1}, which was devised to carry out the above idea for a class of sequences related to modular forms and the combinatorics of integer partitions.

Recall that a {\it partition} of a
non-negative integer $n$ is a weakly decreasing sequence of positive integers
that sum to $n$, and that the {\it partition function} $p(n)$ denotes the number of partitions of $n$.
For example, $p(5)=7$ and the relevant partitions are:
$(5)$, $(4,1)$, $(3,2)$, $(3,1,1)$, $(2,2,1)$, $(2,1,1,1)$, and
$(1,1,1,1,1)$. The function $p(n)$ does not have a closed form, nor does it
satisfy any finite order recurrence. However, its asymptotic behavior was proven by Hardy and Ramanujan \cite{HardyRamanujan1}, who showed that
\begin{align}\label{Eq:PartitionAsymptotic}
p(n)\sim \frac{1}{4\sqrt{3}n}e^{\pi \sqrt{\frac{2n}{3}}}
\qquad\qquad
\mbox{as }n\rightarrow\infty.
\end{align}
In fact, they obtained a much stronger result by introducing what is now known as the Hardy-Ramanujan Circle Method, which uses modular transformations to obtain a divergent series whose truncations approximate $p(n)$ with a very small error (a later refinement of Rademacher \cite{Rademacher1} gave a convergent series representation for $p(n)$).

Ingham \cite{Ingham1} showed that \eqref{Eq:PartitionAsymptotic} can also be derived from a certain Tauberian theorem (see Section \ref{S:Ingham} below).  This approach has recently seen increased use in combinatorics and number theory, including applications in plane partitions \cite{HX}, $t$-core partitions \cite{Rol}, overpartitions \cite{Chen, CKL}, partitions arising from permutation groups \cite{CDF}, families of partitions with certain ``gap'' conditions \cite{KK}, and bounds for the coefficients of modular functions \cite{BrisebarrePhilibert1}. In usage, Ingham's theorem is often stated as follows:
Suppose that $B(q)=\sum_{n\ge0}b_nq^n$ is
a power series with weakly increasing non-negative coefficients and radius
of convergence $1$. If $\lambda$, $\beta$, and $\gamma$ are real numbers with
$\gamma>0$ such that
$B(e^{-t}) \sim \lambda t^\beta e^{\frac{\gamma}{t}}$
as $t\rightarrow0^+$,
then
\begin{gather*}
b_n
\sim
	\frac{ \lambda   \gamma^{\frac{\beta}{2}+\frac14 }}
	{2\sqrt{\pi} n^{\frac{\beta}{2}+\frac34 }}
	e^{2\sqrt{\gamma n}}
\qquad\qquad\mbox{as } n\rightarrow\infty.
\end{gather*}

However, this is not quite correct as written, as it is missing an important technical condition
from Ingham's work.
In particular, the analytic behavior of $B(e^{-z})$ for $z \to 0^+$ along the real axis is not sufficient in general to determine the asymptotic behavior of the coefficients $b_n$, as one also needs to consider $B(e^{-z})$ for complex values of $z$ (see Section \ref{S:Examples:Counterex} below for some counterexamples).
The full statement of Ingham's theorem from \cite{Ingham1} is given in Theorem \ref{Theorem:Ingham} below, and the following result includes all necessary conditions for $B(e^{-z})$. 
The general statement also includes an additional logarithmic term that has 
been needed in some recent applications (see for example \cite{BJSM}).

\begin{theorem}\label{Corollary:CorToIngham}
Suppose that $B(q)=\sum_{n\ge0}b_nq^n$ is a power series with non-negative real
coefficients and radius of convergence at least one. If
$\lambda$, $\alpha$, $\beta$, and $\gamma$ are real numbers with $\gamma>0$ such that
\begin{equation}\label{additional}
B\left(e^{-t}\right) \sim \lambda \log\left( \tfrac{1}{t} \right)^\alpha t^\beta e^{\frac{\gamma}{t}}
\quad\mbox{as } t\rightarrow0^+,
\qquad
B\left(e^{-z}\right) \ll \log\left( \tfrac{1}{|z|} \right)^\alpha |z|^\beta e^{\frac{\gamma}{|z|}}
\quad\mbox{as } z\rightarrow0,
\end{equation}
with $z=x+iy$ ($x,y\in\R, x>0$) in each region of the form
$|y|\le \Delta x$ for $\Delta>0$,
then
\begin{equation}\label{asbsum}
\sum_{n=0}^{N}b_n
\sim
	\frac{ \lambda \gamma^{\frac{\beta}{2}-\frac14 }  \log\left( N \right)^\alpha  }
	{2^{\alpha+1}\sqrt{\pi} N^{\frac{\beta}{2}+\frac14 }}
	e^{2\sqrt{\gamma N}}
\qquad\qquad\mbox{as } N\rightarrow\infty.
\end{equation}
Furthermore, if the $b_n$ are weakly increasing, then
\begin{gather}\label{asb}
b_n
\sim
	\frac{ \lambda\gamma^{\frac{\beta}{2}+\frac14 }  \log\left( n \right)^\alpha  }
	{2^{\alpha+1}\sqrt{\pi} n^{\frac{\beta}{2}+\frac34 }}
	e^{2\sqrt{\gamma n}}
\qquad\qquad\mbox{as } n\rightarrow\infty.
\end{gather}
\end{theorem}

\begin{remarks} 1.
The conclusion of Theorem \ref{Corollary:CorToIngham} forces $B(q)$ to have radius of
convergence exactly one.  We also note that the second condition in \eqref{additional} does not follow from the first using a simple term-by-term estimate, as
\begin{equation*}
\left|B\left(e^{-z}\right)\right|  \leq \sum_{n \geq 0} b_n e^{-n \RE(z)} = B \left(e^{-\RE(z)}\right),
\end{equation*}
but $e^{\frac{\gamma}{\RE(z)}}$ is not $O(e^{\frac{\gamma}{|z|}})$
for complex $z \to 0$.  In fact, in Section \ref{S:Examples:Counterex} we see that the second condition is essential in general.

\noindent 2.
If in each region $|y|\le \Delta x$ we have
\begin{equation}
\label{E:Buniform}
B\left(e^{-z}\right)
\sim \lambda \Log\left( \frac{1}{z} \right)^\alpha z^\beta e^{\frac{\gamma}{z}},
\end{equation}
then the second bound in \eqref{additional} is automatically satisfied.
As explained in Section \ref{Example:Partitions},
this case holds for any modular form with a pole at $z = 0$. Here and throughout we follow the standard convention that $\Log$ denotes the principal branch of the logarithm, so that for $z \neq 0$, $\Log(z) = \log|z| + \Arg(z)$, with $\Arg(z) \in (-\pi, \pi].$
\end{remarks}

The appeal of Ingham's Tauberian theorem is that it yields
asymptotics for sequences with very little effort, particularly in comparison to the Circle Method, which typically requires modular transformations and bounds along various arcs near the complex unit circle. Fortunately, although the bound along the restricted angle $\Delta$ in Theorem \ref{Corollary:CorToIngham} has not always been mentioned explicitly,
the conclusion of the theorem statement still applies in all published applications that we are aware of.
Indeed, one of the purposes of this article is to show that the extra
condition is often guaranteed by whatever method is used to determine the
asymptotic growth of $B(e^{-t})$.
For example, as discussed in Section \ref{S:Ingham},
if the growth is determined by applying transformations of a modular form, then the required bound in the restricted angle is always  satisfied as well.

Another common method for determining the growth of $B(e^{-t})$ is to find an asymptotic expansion of $B(e^{-t})$ for $t$ near zero.
The classical Euler-Maclaurin summation
formula is (see e.g. \cite[equation (2.10.1)]{NIST})
\begin{align*}
\sum_{m=0}^M f(m)
&=
	\int_{0}^M f(x)dx
	+\frac12\left(f(M)+f(0)\right)		
	-
	\sum_{n=1}^{N-1} \frac{ B_{2n}}{(2n)!}\left( f^{(2n-1)}(M) - f^{(2n-1)}(0)  \right)
	\\&\quad
	+
	\int_{0}^M \frac{f^{(2N)}(x)\left( B_{2N}- B_{2N}\left(x-\lfloor x\rfloor\right)\right)}{(2N)!}dx
,
\end{align*}
where $M,N \in \N$,  $B_n(x)$ is the $n$-th Bernoulli polynomial, $B_n$  the $n$-th Bernoulli number,
and $f$ is continuous on the interval $[0,M]$ and $2N$-times continuously
differentiable inside the interval. In \cite{Zagier1}, Zagier gave an elegant account of how this formula implies asymptotic expansions of the form ($N \in \N_0$)
\begin{align}\label{Eq:ZagierEM1}
\sum_{m\ge0} f(t(m+a))
\sim
	\frac{1}{t} \int_{0}^\infty f(x) dx
	-
	\sum_{n=0}^{N-1} \frac{B_{n+1}(a) f^{(n)}(0) }{(n+1)!}t^n + O_N\left(t^N\right),
\end{align}
where $a\in \R^+$ and $f:(0,\infty)\rightarrow\mathbb{C}$ is a $C^\infty$ function such that
$f(x)$ and all of its derivatives are of ``rapid decay'' as $x\rightarrow0$.  For example, this approach has been used to determine the asymptotic behavior of partial theta functions $\sum_{n \geq 0} (-1)^n q^{an^2 + bn}$ as $q \to 1^-$ in \cite{KK,Wright71}.

In consideration of Theorem \ref{Corollary:CorToIngham}, the immediate
question is to what extent do we also have expansions of this form when $f$ is a function of a complex variable.
To be precise, we  say that a function $f$ is of {\it sufficient decay} in a domain $D\subset\C$ if there exists some $\varepsilon > 0$
such that $f(w) \ll w^{-1-\varepsilon}$ as $|w| \to \infty$ in $D$.
Our first result shows that Euler-Maclaurin summation gives an asymptotic expansion that converges {uniformly} on domains that preclude a tangential approach to $0$.

\begin{theorem}\label{Theorem:EulerMaclaurin1DShifted}
Suppose that $0\le \theta < \frac{\pi}{2}$ and let
$D_\theta := \{ re^{i\alpha} : r\ge0 \mbox{ and } |\alpha|\le \theta  \}$.
Let $f:\C\rightarrow\C$ be holomorphic in a domain containing
$D_\theta$, so that in particular $f$ is holomorphic at the origin, and
assume that $f$ and all of its derivatives are of sufficient decay.
Then for $a\in\mathbb{R}$ and $N\in\N_0$,
\begin{gather*}
\sum_{m\geq0}f(w(m+a))
=
	\frac{1}{w}\int_0^\infty f(x) dx
	-
	\sum_{n=0}^{N-1} \frac{B_{n+1}(a) f^{(n)}(0)}{(n+1)!}w^n
	+
	O_N\left(w^N\right)
,
\end{gather*}
uniformly, as $w\rightarrow0$ in $D_\theta$.
\end{theorem}
\begin{remark}
We see in the proof of Theorem \ref{Theorem:EulerMaclaurin1DShifted} that the decay assumption can be slightly relaxed, as the primary technical condition is that $|f^{(n)}(w)|$ is integrable. However, in practice $f(w)$ often has much stronger decay (for example, $f(w) = g(w) e^{-w}$ for a rational function $g$).
\end{remark}

The next result extends Theorem \ref{Theorem:EulerMaclaurin1DShifted} to the case that the function has a simple pole at zero. In order to state it we  need the constants
\begin{gather*}
C_a :=  (1-a) \sum_{m \geq 0} \frac{1}{(m+a)(m+1)},
\end{gather*}
which are defined for
$a\in\mathbb{R}, a \not \in -\N_0$.
We note that
$C_a=-\gamma-\psi(a)$, where $\psi(a):= \frac{\Gamma'(a)}{\Gamma(a)}$
is the {\it digamma function} \cite[equation 6.3.16]{AS64},  and $\gamma$ is the Euler-Mascheroni constant.

\begin{theorem}
\label{Prop:EulerMaclaurin1DPoleShifted}
Suppose that $0\le \theta < \frac{\pi}{2}$, let $f:\C\rightarrow\C$ be holomorphic in a domain containing
$D_\theta$, except for a simple pole at the origin,
and assume that $f$ and all of its derivatives are of sufficient decay in $D_\theta$.
If $f(w) = \sum_{n\geq -1} b_{n}w^n$ near $0$, then for $a\in\mathbb{R}, a \not \in -\N_0$,
and $N\in\N_0$, then uniformly, as $w\rightarrow0$ in $D_\theta$,
\begin{multline*}
\sum_{m\geq0} f(w(m+a))
=
	\frac{b_{-1}\Log\left(\frac{1}{w}\right)}{w}
	+
	\frac{b_{-1}C_a}{w}
	+
	\frac{1}{w}\int_0^\infty \left( f(x) - \frac{b_{-1}e^{-x}}{x}\right) dx
\\
-
	\sum_{n=0}^{N-1} \frac{B_{n+1}(a) b_n}{n+1}w^n +
	O_N\left(w^N\right)
.
\end{multline*}
\end{theorem}

There are also applications where one needs asymptotic expansions of the form \eqref{Eq:ZagierEM1} for sums over multiple indices in $\N$ (e.g. \cite[Section 4]{BringmannJenningsShafferMahlburgRhoades1}). This requires a multi-dimensional generalization of
Theorem \ref{Theorem:EulerMaclaurin1DShifted}. While the two-dimensional version of this formula has appeared in a small number of recent articles, and the authors stated the general form in \cite{BringmannJenningsShafferMahlburgRhoades1}, we are not aware of any recorded proofs. Writing vectors in bold letters and their components with subscripts here and throughout the paper, we say that
a multivariable
function $f$ in $s$
variables is of {\it sufficient decay} in $D$ if there exist $\varepsilon_j>0$ such that $f(\b{x})\ll
(x_1+1)^{-1-\varepsilon_{1}}\dotsm(x_s+1)^{-1-\varepsilon_{s}}$
uniformly as $|x_1|+\dotsb+|x_s|\rightarrow\infty$ in
$D$.

\begin{theorem}\label{Theorem:EulerMaclaurinGeneral}
Suppose that $0 \leq \theta_j < \frac{\pi}{2}$ for $1 \leq j \leq s$,
and that $f:\C^s\rightarrow\C$ is holomorphic in a domain containing $D_{\b\theta}:=D_{\theta_1}\!\times\dotsb\times D_{\theta_s}$. If $f$ and all of its derivatives are of sufficient decay in $D_{\b\theta}$, then for $\b{a}\in\R^s$ and $N\in\N_0$ we have
\begin{align*}
&
\sum_{\b{m}\in \N_0^s } f( w(\b{m}+\b{a}))
=
	(-1)^s\sum_{ \b{n}\in \NN^s}	
	f^{(\b{n} )}( \b{0} )	
	\prod_{1\le j\le s} \frac{B_{n_j+1}(a_j)}{(n_j+1)!}w^{n_j} 	
	\\&\quad
	+
	\sum_{\emptyset\subseteq \mathscr{S} \subsetneq \{1,\dots,s\}}	 
	\frac{(-1)^{|\mathscr{S}|}}{w^{s-|\mathscr{S}|}}
	\sum_{\substack{ n_j\in \mathcal N_N\\ j\in\mathscr{S}}}
	\int_{[0,\infty)^{s-|\mathscr{S}|}}
	\left[
		\prod_{j\in\mathscr{S}}
		\frac{\partial^{n_j}}{\partial x_j^{n_j}}
		f(\b{x})
	\right]_{\substack{x_j=0\\ j\in\mathscr{S}}}
	\prod_{\substack{1\le k\le s\\[0.5ex] k\not\in\mathscr{S}}} dx_k
	\prod_{j\in\mathscr{S}} \frac{B_{n_j+1}(a_j)}{(n_j+1)!}w^{n_j}
	\\&\quad
	+O_N\left(w^N\right)
,
\end{align*}
uniformly, as $w\rightarrow0$ in $D_{\b\theta}$,
where $\mathcal N_N:=\{0,1,\dotsc,N-1\}$.
\end{theorem}

We are writing Theorem \ref{Theorem:EulerMaclaurinGeneral} in a compact form,
so to illustrate the unpacked statement we note that the
two-dimensional case is
\begin{align*}
\sum_{\b{m}\in\N_0^2} f(w(\b{m}+\b{a}))
&=
	\frac{1}{w^2}\int_{0}^{\infty}\!\!\int_{0}^{\infty}\! f(\b{x}) dx_1dx_2
	-
	\frac{1}{w}\sum_{n_1=0}^{N-1}\frac{ B_{n_1+1}(a_1) }{(n_1+1)!} w^{n_1}
		\int_{0}^{\infty}\!  f^{(n_1,0)} (0,x_2)dx_2
	\\&\quad
	-
	\frac{1}{w}
	\sum_{n_2=0}^{N-1} \frac{B_{n_2+1}(a_2) }{(n_2+1)!}w^{n_2}
	\int_{0}^{\infty}\! f^{(0,n_2)}(x_1,0)dx_1
	\\&\quad
	+
	\!\sum_{n_1+n_2<N}\!\!
	 \frac{B_{n_1+1}(a_1) B_{n_2+1}(a_2) f^{(n_1,n_2)}(\b{0})}{(n_1+1)!(n_2+1)!}w^{n_1+n_2}
	+
	O_N\left(w^N\right)
.
\end{align*}

The remainder of this article is organized as follows.
In the following section we recall known facts for Bernoulli polynomials.
In Section \ref{S:Examples}, we give examples of a few applications related to
Theorems \ref{Corollary:CorToIngham} and \ref{Theorem:EulerMaclaurin1DShifted}.
In particular, we demonstrate why the additional growth constraint
in the right half-plane is necessary for Theorem \ref{Corollary:CorToIngham} and why
the limit in Theorem \ref{Theorem:EulerMaclaurin1DShifted} must be taken non-tangentially
to the imaginary axis.
In Section \ref{S:Ingham}, we state Ingham's Tauberian theorem and use it to prove Theorem
\ref{Corollary:CorToIngham}. In Section \ref{S:Euler-Mac} we extend the classical use of Euler-Maclaurin summation to complex functions, proving Theorems \ref{Theorem:EulerMaclaurin1DShifted}, \ref{Prop:EulerMaclaurin1DPoleShifted},
and \ref{Theorem:EulerMaclaurinGeneral}.
We conclude in Section \ref{S:Conclusion} with a brief discussion comparing Ingham's Tauberian theorem and Wright's Circle Method.

\section*{Acknowledgments}
The authors thank Karl Dilcher for bringing Euler-Boole summation to our attention. Moreover we thank the referee for useful comments on an earlier version of this paper.

\section{Preliminaries}
\label{S:Prelim}


We begin by recalling basic properties of the \textit{Bernoulli polynomials} (see \cite[Section 23.1]{AS64}), which have the exponential generating function
\begin{equation*}
\sum_{n \geq 0} B_n(x) \frac{t^n}{n!} = \frac{t e^{tx}}{e^t - 1}.
\end{equation*}
For $n \in\N_0\!\setminus\!\{1\}$, the Bernoulli numbers are defined by
\begin{equation}
\label{E:BnDef}
B_n := B_n(1) = B_n(0),
\end{equation}
whereas in order to avoid confusion for $n=1$,
we typically simply directly plug in
\begin{equation}
\label{E:B1}
B_1(1)=\frac12=-B_1(0).
\end{equation}
The polynomials satisfy many useful identities, including
\begin{align}
\label{E:B'}
B^\prime_{n+1}(x) & =(n+1)B_n(x), \quad \text{and}\\
\label{E:Bkx+y}
B_{k}(x+y) &= \sum_{n=0}^k \binom{k}{n} B_{n}(x)y^{k-n}.
\end{align}

We also need the \textit{Euler polynomials}, which have the generating function
\begin{equation}\label{Euler}
\sum_{n \geq 0} E_n(x) \frac{t^n}{n!} = \frac{2 e^{tx}}{e^t + 1}.
\end{equation}
These are related to the Bernoulli polynomials by the identity
\begin{equation}
\label{E:Bn=En}
B_{n+1}\left(\frac{x}{2}\right) - B_{n+1}\left(\frac{x}{2} + \frac{1}{2}\right)
= -\frac{(n+1)}{2^{n+1}}E_{n}(x).
\end{equation}

\section{Examples}\label{S:Examples}
In this section we consider some applications for Theorems \ref{Corollary:CorToIngham}
and \ref{Theorem:EulerMaclaurin1DShifted}. In these examples we carefully
examine the technical issues that arise in applying and using these theorems.

\subsection{Partitions and weakly holomorphic modular forms}\label{Example:Partitions}
In order to illustrate the use of Theorem \ref{Corollary:CorToIngham}, we first revisit one of the motivating examples in \cite{Ingham1}. Euler's partition generating function is
\begin{gather*}
P(q) := \sum_{n\ge0} p(n)q^n = \frac{q^{\frac{1}{24}}}{\eta(\tau)},
\end{gather*}
where $\eta(\tau):=q^{\frac1{24}}\prod_{n\geq 1}(1-q^n)$ is \emph{Dedekind's $\eta$-function}. Here, and in the other examples, $q$ and $\tau$ are related by $q:=e^{2\pi i\tau}$.
The Dedekind $\eta$-function satisfies the modular transformation \cite[Theorem 3.1]{Apo90}
\begin{gather*}
\eta\left(-\frac{1}{\tau}\right)
=
\sqrt{-i\tau} \eta(\tau),
\end{gather*}
which implies that for $z\in\C$ with $\RE(z) > 0$,
\begin{gather*}
P(e^{-z})
=
	\sqrt{\frac{z}{2\pi}}
	e^{-\frac{z}{24} + \frac{\pi^2}{6z}  }
	\sum_{n\ge0} p(n)e^{-\frac{4\pi^2 n}{z}}
=
	\sqrt{\frac{z}{2\pi}}
	e^{\frac{\pi^2}{6z}}\left(1 + O\left(\left|e^{-\frac{4 \pi^2}{z}}\right|\right)\right).
\end{gather*}
We now easily see that Theorem \ref{Corollary:CorToIngham} can be applied, since if $z=x+iy$ $(x>0)$ with $|y|\le \Delta x$, then
\begin{equation}
\label{E:PExpn}
\left| e^{-\frac{1}{z}} \right|
=
	e^{-{\frac{x}{x^2+y^2}}}
	\le
	e^{-{\frac{1}{\left(1+\Delta^2\right)x}}}
	\le
	e^{-{\frac{1}{\left(1+\Delta^2\right)|z|}}}
.
\end{equation}
Thus in these regions of restricted angle, we have (see \eqref{E:Buniform})
\begin{equation}
\label{E:PAsymp}
P(e^{-z})
\sim
	\sqrt{\frac{z}{2\pi}}
	e^{\frac{\pi^2}{6z}  }	
,\qquad\qquad\qquad
\mbox{as }
z\rightarrow 0.
\end{equation}
And indeed, Theorem \ref{Corollary:CorToIngham} does give the correct asymptotic formula, as \eqref{asb} implies \eqref{Eq:PartitionAsymptotic}.

Finally, we also note that \eqref{E:PAsymp} does not hold in the whole right half-plane $\RE(z) > 0$.
For example, if $z$ approaches $0$ tangentially along the path $z=x+ix^{\frac13}$, then
\begin{gather*}
\exp\left( -\frac{1}{z} \right)
=
	\exp\left( -\frac{x}{x^2+x^{\frac23}} +\frac{ix^{\frac13}}{x^2+x^{\frac23}} \right)
,
\end{gather*}
and
\begin{gather}\label{Eq:TangentialPath}
\frac{x}{x^2+x^{\frac23}} \rightarrow 0
,\qquad\qquad
\frac{x^{\frac13}}{x^2+x^{\frac23}}
\rightarrow \infty,
\end{gather}
as $x\rightarrow0^+$. Thus $|e^{-\frac1z}|\rightarrow1$
and we can no longer isolate the leading asymptotic term in \eqref{E:PExpn}.
This can also be seen numerically. In Table \ref{Table:Errors1},
we give a numerical approximation of the size of
$P(e^{-z}) \sqrt{\frac{2\pi}{z}}e^{-\frac{\pi^2}{6z}}-1$
along three different paths, with the first two being non-tangential and the third being tangential.
As expected, the error tends to zero for the first two paths, but not the third.

The principle is similar for any other modular form, as the modular inversion map $\tau\mapsto -\frac{1}{\tau}$
always gives an expansion in terms of $e^{-cz}$ for some $c>0$, which rapidly
tends to $0$ so long as the angle of $z$ is restricted.

\begin{center}
\begin{table}[h]
\caption{Approximate size of the error, $P(e^{-z})\sqrt{\frac{2\pi}{z}}e^{-\frac{\pi^2}{6z}}-1$, along three paths}
\begin{tabular}{c || r|r|r}\label{Table:Errors1}
\diagbox[height=0.75cm, width=2cm]{$\hspace{8pt}x$}{$\mathrm{path}\hspace{-4pt}$}
& $z=x+ix$ & $z=x+ix^2$ & $z=x+ix^{\frac13}$
\\
\hline \hline &&&
\\[-2ex]
$10^{-1}$ 	& 0.0058802931 	& 0.0041787363  & 0.0197422414 \\
$10^{-2}$ 	& 0.0005891329	& 0.0004166007  & 0.0088566903 \\
$10^{-3}$ 	& 0.0000589243	& 0.0000416658  & 0.0234673077\\
$10^{-4}$ 	& 0.0000058925	& 0.0000041666  & 0.1284298533\\
$10^{-5}$ 	& 0.0000005892	& 0.0000004166  & 0.2648476442
\end{tabular}
\end{table}
\end{center}

\subsection{A counterexample for the real-analytic version of Ingham's theorem.}
\label{S:Examples:Counterex}

In this section we give a detailed presentation of an example that demonstrates the necessity of the second condition in \eqref{additional}.

\subsubsection{General discussion}

The importance of the example under consideration was highlighted by Ingham, who stated in note 2) on page 1088 of \cite{Ingham1} that ``In Theorem $1'$ we may regard\dots (ii) (for every $\Delta$) as Tauberian conditions which convert the generally false inference\dots into a true proposition. An example\dots has been constructed by Avakumovi\'c and Karamata (353, e).''

More specifically, we  work with the special case $\gamma = \frac32$ of Avakumovi\'c and Karamata's example e) \cite{AK36} (after making some minor modifications in order to obtain a power series instead of the continuous Laplace transform used in their original construction).

We define the coefficients
$$
A(n) := \begin{cases} 0 & \text{if } n = 0, \\
e^{2 m^{\frac32}} m^{-\frac14} \qquad & \text{if } m^{3} \leq n < (m+1)^{3},
\end{cases}
$$
and the corresponding series $F(q) := \sum_{n \geq 0} A(n) q^n$.
\begin{proposition}
\label{P:AK}
As $t \to 0^+$,
\begin{equation*}
F\left(e^{-t}\right) \sim \frac{2\sqrt{\pi}}{3} \frac{e^{\frac1{t}}}{t},
\end{equation*}
and
\begin{align}
\label{E:Anlimsup}
\limsup_{n \to \infty} n^{\frac{1}{12}} e^{-2 \sqrt{n}} A(n) & = 1, \\
\label{E:Anliminf}
\liminf_{n \to \infty} n^{\frac{1}{12}} e^{-2 \sqrt{n} + 3 n^{\frac16}} A(n) & = 1.
\end{align}
\end{proposition}

As a consequence, we see that Theorem \ref{Corollary:CorToIngham}
 is false in general if one only considers the asymptotic behavior of the series along the real line.
 In particular, the $A(n)$ are weakly increasing and $F(e^{-t})$ satisfies the first condition in \eqref{additional}. However, a short calculation shows that \eqref{asb} does not hold; otherwise, the conclusion would be that $A(n) \sim B(n)$, with
\begin{equation*}
B(n) := \frac{1}{3} n^{-\frac14} e^{2 \sqrt{n}}.
\end{equation*}
But \eqref{E:Anlimsup} and \eqref{E:Anliminf} show that this expression does not accurately describe the behavior of $A(n)$, either from above or below, as
\begin{equation*}
\limsup_{n \to \infty} \frac{A(n)}{B(n)}
= 3 \limsup_{n \to \infty} n^{\frac16} = \infty, \qquad
\liminf_{n\to \infty} \frac{A(n)}{B(n)}
= 3 \liminf_{n \to \infty} n^{\frac16} e^{- 3 n^{\frac16}} = 0.
\end{equation*}

\subsubsection{Proof of Proposition \ref{P:AK}}

We first verify the asymptotic formulas for the coefficients. By construction, if $s(n)$ is an increasing sequence, then any maxima of $\frac{A(n)}{s(n)}$ occur at the values $n = m^3$, thus
\begin{align*}
\limsup_{n \to \infty} n^{\frac{1}{12}} e^{-2 \sqrt{n}} A(n)
= \limsup_{m \to \infty} m^{\frac14} e^{-2 \sqrt{m^3}} A\left(m^3\right) = 1.
\end{align*}
This proves \eqref{E:Anlimsup}.

Similarly, the minima of $\frac{A(n)}{B(n)}$ occur at $n = (m+1)^3-1$. To simplify the calculations, note that the expression $n^{\frac{1}{12}}e^{-2 \sqrt{n} + 3 n^{\frac16} }$ is asymptotically equal to the same expression with $n \mapsto n-1.$ We can therefore plug in $(m+1)^3$ instead of $(m+1)^3-1$. Thus
\begin{align}
\notag
\liminf_{n \to \infty} e^{-2 \sqrt{n} + 3 n^{\frac16}} n^{\frac{1}{12}} A(n)
& = \liminf_{m \to \infty} e^{-2(m+1)^{\frac32} + 3 (m+1)^{\frac12}} (m+1)^{\frac14} A\left(m^3\right) \\ \notag
& = \liminf_{m \to \infty} e^{-2(m+1)^{\frac32} + 3 (m+1)^{\frac12}} (m+1)^{\frac14}  e^{2m^{\frac32}} m^{-\frac14}.
\end{align}
As $m \to \infty$, we have that $(m+1)^{\frac14} m^{-\frac14} \to 1$. The exponential term has the overall exponent
\begin{equation*}
-2(m+1)^{\frac32}  + 3(m+1)^{\frac12} + 2m^{\frac32}
 = O\left(m^{-\frac12}\right),
\end{equation*}
which goes to 0 as $m\to \infty$. This proves \eqref{E:Anliminf}.

It is more involved to determine the asymptotic behavior of $F$. By definition,
\begin{equation} \label{E:f2sums}
F\left(e^{-t}\right)
= \sum_{m \geq 1} e^{2m^{\frac32}} m^{-\frac14} \sum_{n=m^3}^{(m+1)^3-1} e^{-nt}
= \frac{1}{1-e^{-t}} \sum_{m \geq 1} e^{2m^{\frac32}} m^{-\frac14} \left(e^{-m^3 t} - e^{-(m+1)^3 t}\right).
\end{equation}

We see below that the final exponential term is asymptotically negligible, and we next show that when this final term is removed, the resulting sum indeed gives the claimed asymptotic formula.
\begin{proposition}
\label{P:gsumAsymp}
As $t \to 0^+$, we have
\begin{equation*}
\sum_{m \geq 1} m^{-\frac14}  e^{2m^{\frac32}-m^3 t}
\sim e^{t^{-1}}\int_1^\infty u^{-\frac14} e^{-t \left(u^{\frac32}- \frac1{t}\right)^2} du
\sim \frac{2\sqrt{\pi}}{3} e^{\frac1{t}}.
\end{equation*}
\end{proposition}
\begin{proof}
We roughly follow the arguments on pages 354--355 of \cite{AK36}, with some additional details added. We begin by showing the final asymptotic equality, as it is useful throughout the rest of the proof. Using the substitution $w = \sqrt{t} (u^{\frac32} - \frac1{t})$, the integral becomes
\begin{align*}
\int_{1}^\infty u^{-\frac14} e^{-t\left(u^{\frac32}-\frac1{t}\right)^2} du
&= \frac{2}{3\sqrt{t}} \int_{\sqrt{t}-\frac1{\sqrt{t}}}^\infty \left(\frac{w}{\sqrt{t}}+\frac1{t}\right)^{-\frac12} e^{-w^2} dw
\overset{t\to 0}{\to} \frac{2}{3} \int_{-\infty}^{\infty} e^{-w^2} dw = \frac{2\sqrt{\pi}}{3}.
\end{align*}

For the sum, noting that $e^{2m^\frac32 - m^3 t} = e^{\frac1{t}} e^{-t(m^\frac32 - \frac1{t})^2}$, we approximate $\sum_{m \geq 1} g(m^{\frac32})$, where
\begin{equation*}
g(x) := x^{-\frac16} e^{-t \left(x-\frac1{t}\right)^2}.
\end{equation*}
We prove the integral approximation by showing that the summands $g(m)$ are unimodal, with a peak near $m \sim t^{-\frac23}$. The growth rate of these terms is determined by the derivative of $g(x)$, which is
\begin{equation*}
g'(x)= 2t x^{-\frac76} e^{-t \left(x - \frac1{t}\right)^2} \left(-x^2 + \frac{x}{t} - \frac{1}{12t}\right).
\end{equation*}
The terms in front are always positive for $x > 0$, so the sign of $g'(x)$ is determined by the quadratic expression in the parentheses. The roots of this expression are
\begin{equation}
\notag
x_1 =\frac{1}{2t} \left(1- \sqrt{1 - \frac{t}{3}}\right), \qquad
x_2 = \frac{1}{2t} \left(1+ \sqrt{1 - \frac{t}{3}}\right);
\end{equation}
the minimum of $g$ occurs at $x_1$, and the maximum at $x_2$.

However, the behavior near $x_1$ does not have any effect on $F(t)$, since, as $t\to 0$ $x_1 \sim  \frac{1}{12}$. Specifically, this means that for sufficiently small $t$, the minimum of $g(x)$ occurs at some $x < 1$, and thus the sum beginning at $m = 1$ is monotonically increasing until $m_2 := \lfloor x_2^{\frac23} \rfloor$, and monotonically decreasing beginning from $m_2 + 1$.
Moreover, we need the observation that
$
x_2\sim\frac1{t}.
$

The standard integral comparison criterion for monotonic functions now implies that
\begin{gather*}
\int_{1}^{m_2} g\left(x^{\frac32}\right)dx
<
	\sum_{m=1}^{m_2} g\left(m^{\frac32}\right)
	<	
	\int_{1}^{m_2} g\left(x^{\frac32}\right)dx
	+ g\left(x_2\right)
,\\
\int_{m_{2}+1}^{\infty} g\left(x^{\frac32}\right)dx
<
	\sum_{m=m_2+1}^{\infty} g\left(m^{\frac32}\right)
	<	
	\int_{m_2+1}^\infty g\left(x^{\frac32}\right)dx
	+ g\left(x_2\right)
.
\end{gather*}
From this, we obtain that
\begin{align*}
\int_{1}^{\infty} g\left(x^{\frac32}\right)dx
&<
	\int_{1}^{m_2} g\left(x^{\frac32}\right)dx
	+
	g\left(x_2\right)
	+
	\int_{m_2+1}^{\infty} g\left(x^{\frac32}\right)dx
\\
&<
	\sum_{m=1}^{\infty} g\left(m^{\frac32}\right)
	+
	g\left(x_2\right)
<
	\int_{1}^{\infty} g\left(x^{\frac32}\right)dx
	+ 3g(x_2).
\end{align*}
Thus
\begin{equation}
\label{E:gsum-int}
\left|\sum_{m \geq 1} g\left(m^\frac32\right) - \int_1^\infty g\left(x^\frac32\right) dx\right|
< 2 g\left(x_2\right).
\end{equation}

Using that $g(x_2)=o(1)$ and the integral evaluation
$\int_1^\infty g(x^\frac32) dx \sim \frac{2 \sqrt{\pi}}{3}$,
the bound in \eqref{E:gsum-int}
shows that the sum and integral are asymptotically equal.
\end{proof}

We now prove that the final sum in \eqref{E:f2sums} is asymptotically smaller than the remaining terms.
\begin{lemma}[\cite{AK36}, p. 350]
\label{L:sumQuotient}
As $t \to 0^+$, we have
$$
\frac{\sum_{m \geq 1} m^{-\frac14} e^{2m^{\frac32} - (m+1)^3 t}}{\sum_{m \geq 1} m^{-\frac14} e^{2m^{\frac32} - m^3 t}} = o(1).
$$
\end{lemma}
\begin{proof}
We show that the sum in the numerator is termwise smaller than the denominator for $m > t^{-\frac12 - \varepsilon}$ (for some $\varepsilon > 0$), and the sum over this initial segment of $m$s is itself asymptotically negligible. More precisely, since $g(m^\frac32)$ is increasing in this range we obtain for $0<\varepsilon<\frac{1}{6}$,
\begin{align}
\label{E:gInitSum}
 \sum_{m = 1}^{\left\lfloor t^{-\frac12 - \varepsilon} \right\rfloor} g\left(m^\frac32\right)
&\leq t^{-\frac12 - \varepsilon} g\left(t^{-\frac32\left(\frac12 + \varepsilon\right)}\right) =
t^{-\frac38(1 + 2 \varepsilon)} e^{-t \left(t^{-\frac34(1+2\varepsilon)} - \frac1{t}\right)^2}=o(1).
\end{align}

This also gives
\begin{equation}
\notag
\sum_{m \geq 1} g\left(m^{\frac32}\right) \sim \sum_{m > t^{-\frac12 - \varepsilon}} g\left(m^{\frac32}\right),
\end{equation}
since Proposition \ref{P:gsumAsymp} shows that the left-hand side is asymptotically $\frac{2 \sqrt{\pi}}{3}$.

Continuing, since each term in the numerator of the lemma statement is smaller than the corresponding denominator term,
\eqref{E:gInitSum} also implies that
\begin{equation*}
\sum_{m \geq 1} m^{-\frac14} e^{2m^{\frac32} - (m+1)^3 t}
= o(1) + \sum_{m > t^{-\frac12 - \varepsilon}} m^{-\frac14} e^{2m^{\frac32} - (m+1)^3 t}.
\end{equation*}
The final sum can then be compared termwise to the denominator sum, that is,
\begin{align*}
\sum_{m > t^{-\frac12 - \varepsilon}}  m^{-\frac14} e^{2m^\frac32 - (m+1)^3 t}
&= \sum_{m > t^{-\frac12 - \varepsilon}} m^{-\frac14} e^{2m^\frac32 - m^3 t} e^{-\left(3m^2 + 3m + 1\right)t}
\\
&< e^{-3t^{-2\varepsilon}} \sum_{m > t^{-\frac12 - \varepsilon}} m^{-\frac14} e^{2m^\frac32 - m^3 t}.
\end{align*}
Using \eqref{E:gInitSum} and Proposition \ref{P:gsumAsymp} gives that
\begin{gather*}
\frac{\sum_{m \geq 1} m^{-\frac14} e^{2m^{\frac32} - (m+1)^3 t}}{\sum_{m \geq 1} m^{-\frac14} e^{2m^{\frac32} - m^3 t}}
=
	\frac{o(1) + \frac{2\sqrt{\pi}}{3}e^{-3t^{-2\varepsilon}+\frac{1}{t}}  }{o(1) + \frac{2\sqrt{\pi}}{3} e^{\frac{1}{t}} }
=
o(1).\qedhere
\end{gather*}

\end{proof}

Finally, the proof of Proposition \ref{P:AK} is completed by combining Proposition \ref{P:gsumAsymp} and Lemma \ref{L:sumQuotient}. In particular, by plugging these in to \eqref{E:f2sums}, we find that the main asymptotic term is
\begin{align}
\notag
f(t) & \sim \frac{1}{1-e^{-t}} \sum_{m \geq 1} m^{-\frac14} e^{2 m^{\frac32} - m^3 t}
\sim  \frac{2 \sqrt{\pi}}{3t} e^{\frac1{t}}.
\end{align}

\subsection{Eisenstein series}
Here we give an example to demonstrate that the expansion in Theorem \ref{Theorem:EulerMaclaurin1DShifted}
may fail when $w$ is allowed to approach $0$ along tangential paths in the right half-plane.
For this we use the  Eisenstein series of weight four for the full modular group. However,
examples of this type generally arise from any modular form to which
Theorem \ref{Theorem:EulerMaclaurin1DShifted} can be applied.
Set
\begin{align*}
E_4(\tau)
&:=
	1+240\sum_{n\ge1} \frac{n^3 q^n}{1-q^n}
=
	1+240\sum_{n\ge1} \sigma_3(n)q^n
,\\
g_3(q)
&:=
	\sum_{n\ge1} \sigma_3(n) q^n	
=
	\sum_{n\ge1}\frac{n^3q^n}{1-q^n}	
.
\end{align*}
From the modular transformation,
\begin{gather*}
E_4\left(-\frac{1}{\tau}\right)
=
	\tau^{4} E_4(\tau),
\end{gather*}
for $\IM(\tau)>0$, we find that, for $\RE(w)>0$,
\begin{align}\label{Eq:EisensteinEulerMacluarinProblem1}
g_3\left(e^{-w}\right)
&=
	\frac{E_4\left(\frac{iw}{2\pi}\right)-1}{240}
=
	\frac{\left(\frac{2\pi}{w}\right)^4 E_4\left(\frac{2\pi i}{w}\right)-1}{240}
=
	\left(\frac{2\pi}{w}\right)^{4}\left( g_3\left(e^{-\frac{4\pi^2}{w}}\right)+\frac{1}{240}  \right)
	-
	\frac{1}{240}.
\end{align}
Thus when $w\rightarrow0$ on paths non-tangential to the imaginary
axis, we have for each $N\in\N_0$ that
\begin{align}\label{Eq:EisensteinEulerMacluarinProblem2}
g_3(e^{-w})
&=
	\frac{\pi^{4}}{15w^4} - \frac{1}{240} + O_N\left(w^N\right)
.
\end{align}
As explained in Example 3 of \cite{Zagier1}, one can also deduce \eqref{Eq:EisensteinEulerMacluarinProblem2} directly
from Theorem \ref{Theorem:EulerMaclaurin1DShifted} by taking
$f(w):=\frac{w^3 e^{-w}}{1-e^{-w}}$ and $a=1$, writing
\begin{gather*}
g_3(e^{-w}) =
\sum_{n\ge1} \frac{n^3e^{-wn}}{1-e^{-wn}}
=
\frac{1}{w^{3}}\sum_{m\ge0} f(w(m+1)).
\end{gather*}
However, along paths
tangential to the imaginary axis, \eqref{Eq:EisensteinEulerMacluarinProblem2}
may fail.
For example, along the path $w=x+ix^{\frac13}$,
\eqref{Eq:TangentialPath} shows that every point along the unit circle
is a limit point of $e^{-\frac{4\pi^2}{w}}$.
In particular, since $g_3(q)\rightarrow \infty$ as $q\rightarrow1^-$, we see
that
$\limsup_{w\rightarrow0} | g_3(e^{-\frac{4\pi^2}{w}}) | = \infty$
on the path $w=x+ix^\frac13$. Thus
\eqref{Eq:EisensteinEulerMacluarinProblem1} implies that
\eqref{Eq:EisensteinEulerMacluarinProblem2} cannot hold
along this path. This gives an example where Theorem \ref{Theorem:EulerMaclaurin1DShifted}
fails without the additional assumption that $w\in D_\theta$.
Again this is visible from numerical data.
In Table \ref{Table:Errors2},
we give a numerical approximation of the size of
$g_3(e^{-w}) - \frac{\pi^{4}}{15w^4} + \frac{1}{240}$
along three different paths.

\begin{center}
\begin{table}[h]
\caption {Approximate size of the error $g_3(e^{-w}) - \frac{\pi^{4}}{15w^4} + \frac{1}{240}$. }
\begin{tabular}{c || c|c|c}\label{Table:Errors2}
\diagbox[height=0.75cm, width=2cm]{$\hspace{5pt}x$}{$\hspace{8pt}\mathrm{\hspace{-5pt}path}\hspace{-4pt}$}
& $w=x+ix$ & $w=x+ix^2$ & $w=x+ix^{\frac13}$
\\
\hline \hline &&&
\\[-2ex]
$10^{-1}$ & $0.18293\cdot 10^{-55}$		& $0.67139\cdot 10^{-139}$ 	 &  $19030\cdot 10^{16}$\\
$10^{-2}$ & $0.53168\cdot 10^{-823}$		& $0.17223\cdot 10^{-1679}$ 	&  $37122\cdot 10^{21}$\\
$10^{-3}$ & $0.22863\cdot 10^{-8534}$		& $0.22329\cdot 10^{-17106}$ 	&  $75858\cdot 10^{24}$\\
$10^{-4}$ & $0.49431\cdot 10^{-85684}$		& $0.10073\cdot 10^{-171409}$ 	&  $12065\cdot 10^{27}$\\
$10^{-5}$ & $0.11030\cdot 10^{-857216}$	& $0.49985\cdot 10^{-1714479}$ &  $58757\cdot 10^{28}$
\end{tabular}
\end{table}
\end{center}

\subsection{Asymptotic expansions valid along any path}
In contrast to the previous example, there are also functions where the asymptotic expansion of Theorem
\ref{Theorem:EulerMaclaurin1DShifted} is valid on general paths.
For example, taking $f(w):=e^{-w}$
and $a=0$ in Theorem \ref{Theorem:EulerMaclaurin1DShifted} gives
\begin{align}\label{Eq:EulerMaclaurinNoProblem}
\sum_{m\geq0} e^{-wm}
&=
	\frac{1}{w} - \sum_{n=0}^{N-1} \frac{(-1)^n B_{n+1}(0)}{(n+1)!}w^n + O_N\left(w^N\right)	
,
\end{align}
as $w\rightarrow0$ in any $D_\theta$.
The left-hand side
of \eqref{Eq:EulerMaclaurinNoProblem} can be summed as a geometric series when $\RE(w)>0$,
\begin{gather*}
\sum_{m\geq 0} e^{-wm}
=
	\frac{1}{1-e^{-w}}
.
\end{gather*}
The right hand-side of \eqref{Eq:EulerMaclaurinNoProblem} can be interpreted in terms of a truncation
of the generating function for Bernoulli numbers, specifically
\begin{gather*}
\frac{1}{w} - \sum_{n\geq0} \frac{(-1)^n B_{n+1}(0)}{(n+1)!}w^n
=
	\frac{1}{w}\sum_{n\geq0} \frac{ B_{n}(0)}{n!}(-w)^n
=
	\frac{1}{1-e^{-w}}
.
\end{gather*}
From this we see that \eqref{Eq:EulerMaclaurinNoProblem} is valid with $w\rightarrow0$
along any path, as $\frac{w}{1-e^{-w}}$ is analytic in $|w| < 2\pi$

That the asymptotic expansion of Theorem \ref{Theorem:EulerMaclaurin1DShifted} is valid
for general paths to $0$ for some functions and not others should come as no surprise.
The series $\sum_{m\ge0}f(w(m+a))$ defines a holomorphic function in the right
half-plane and the series diverges at $w=0$. The analytic behavior of such a function can
be anything from a simple pole at $w=0$ to the imaginary axis being a natural boundary
in the sense that the singularities are dense along the axis.

\section{Ingham's Tauberian theorem and the proof of Theorem \ref{Corollary:CorToIngham}}
\label{S:Ingham}

In this section we discuss Ingham's Tauberian theorem and use it to prove Theorem \ref{Corollary:CorToIngham}.
We start by recalling the following theorem, which is due to Ingham \cite[Theorem 1]{Ingham1}.
\begin{theorem}\label{Theorem:Ingham}
Let $D$ be a connected open subset of $\mathbb{C}$ that contains
$(0,h]$ (for some $h\in\R^+$). For $t\in(0,h]$, we let $\delta(t)$ denote the distance from $t$ to
the complement of $D$.  Suppose that $\varphi$ and $\chi$ are functions on $D$ that satisfy the following conditions:
\begin{enumerate}[leftmargin=*, label={\rm(\alph*)}]
\item  $\varphi$ and $\chi$ are
holomorphic on $D$, and are positive on $(0,h]$;
\item $-t\varphi^\prime(t) \to \infty$ as $t \to 0^+$;
\item $-\dfrac{\delta(t) \varphi^\prime(t)}{t\sqrt{\varphi^{\prime\prime}(t)}}\rightarrow\infty$ as $t \to 0^+$, and
\item $\varphi^{\prime\prime}(t+z) = O\left( \varphi^{\prime\prime}(t) \right)$ and
	$\chi(t+z) = O\left( \chi(t) \right)$ uniformly in $z$ for $|z|<\delta(t)$ as $t \to 0^+.$
\end{enumerate}
Let $A:[0,\infty)\rightarrow\mathbb{R}$ be a weakly increasing function with $A(0)=0$. Set
\begin{align*}
f(z) := \int_0^\infty e^{-zu}dA(u),
\end{align*}
and assume that $f(z)$ exists for $\RE(z)>0$. Suppose that the following conditions hold:
\begin{enumerate}[leftmargin=\widthof{\rm(ii)}+\labelsep, label={\rm(\roman*)}]
\item $f(z)\sim \chi(z)e^{\varphi(z)}$ as $z\rightarrow0$ with $z$ in $D$;
\item for each fixed $\Delta>0$, $f(z)=O(\chi(|z|)e^{\varphi(|z|)})$ as $z\rightarrow0$
where  $|\IM(z)|\le \Delta\RE(z)$.
\end{enumerate}
Then
\begin{gather*}
A(x)
\sim
	\frac{\chi(\psi(x)) e^{ \varphi(\psi(x)) + x\psi(x) } }{\psi(x) \sqrt{2\pi \varphi^{\prime\prime}(\psi(x))  }}
,\qquad\qquad \mbox{as } x\rightarrow\infty,
\end{gather*}
where $\psi$ is the inverse function of $-\varphi^\prime$.
\end{theorem}

Ingham also discussed a special case that eliminates the need for calculating an exact asymptotic formula for $f$ throughout the complex domain $D$ (although it is still necessary to bound the asymptotic order of $f$ as in condition (ii)). The following result is \mbox{\cite[Theorem $1'$]{Ingham1}}.

\begin{theorem}\label{Theorem:InghamPrime}
Suppose that conditions {\rm(a)},\,{\rm(b)}, {\rm(c)}, and {\rm(d)} of Theorem \ref{Theorem:Ingham}
are satisfied, and additionally, $-t^k\varphi^\prime(t)$ decreases to $0$ for some fixed $k\in\R$. Then condition {\rm (i)} may be replaced by
\begin{enumerate}[leftmargin=*, label={\rm(\alph*)}]
\item[{\rm (i$'$)}]
$f(t)\sim \chi(t)e^{\varphi(t)}$ as $t\rightarrow0^+$.
\end{enumerate}
\end{theorem}

We are now ready to prove Theorem \ref{Corollary:CorToIngham}.

\begin{proof}[Proof of Theorem \ref{Corollary:CorToIngham}]
We first show that \eqref{asb} follows from \eqref{asbsum} applied to the series $C(q):=(1-q)B(q)$.
The monotonicity of $b_n$ implies that $C(q)$ has non-negative coefficients.
 Furthermore, as $z \to 0$ we have
\begin{equation*}
C(e^{-z})=(1-e^{-z})B(e^{-z})\sim zB(e^{-z}).
\end{equation*}

The theorem is trivially true for $\lambda=0$, so we assume $\lambda>0$.
To prove \eqref{asbsum}, we apply Theorems \ref{Theorem:Ingham} and \ref{Theorem:InghamPrime}
with
\begin{gather*}
\varphi(z) := \frac{\gamma}{z}
,\qquad\qquad
\chi(z) := \lambda\operatorname{Log}\left(\frac{1}{z}\right)^\alpha z^{\beta}
,\qquad\qquad
A(x) := \sum_{n<x} b_n.
\end{gather*}
We let $D$ consist of those
points $z$ which satisfy $|{\rm Arg}(z)|< \frac{\pi}{4}$. In particular,
in this region $\delta(t)=\frac{t}{\sqrt{2}};$  this follows from applying the Law of Sines to calculate the distance from $t$ to the ray along ${\rm Arg}(z) = \frac{\pi}{4}.$
To verify that the conditions on $\varphi$ and $\chi$ are satisfied, we differentiate
\begin{align*}
\varphi^\prime(z) = -\frac{\gamma}{z^2}
,\qquad\qquad
\varphi^{\prime\prime}(z)
=\frac{2\gamma}{z^3}
, \qquad\qquad
-\frac{\delta(t)\varphi^\prime(t)}{t\sqrt{\varphi^{\prime\prime}(t)}}
=	
	\frac{\sqrt{\gamma}}{2t^{\frac12}}
.
\end{align*}
It is then obvious that conditions (a), (b), and (c) hold,
as well as the extra growth condition from Theorem \ref{Theorem:InghamPrime}.
If $|z| < \frac{t}{\sqrt{2}}$ then
$(1-\frac{1}{\sqrt{2}})t < |z+t| < (1+\frac{1}{\sqrt{2}})t$, and and thus as $t\rightarrow0^+$ we have
$\varphi^{\prime\prime}(z+t)=O(\varphi^{\prime\prime}(t))$
uniformly in $z$. Furthermore,
\begin{gather*}
\left| \Log\left( \frac{1}{z+t}  \right) \right| = \left|\log\left(\frac{1}{|z+t|}\right)+i{\rm Arg}(z+t)\right|
\sim
	\left| \log\left( \frac{1}{|z+t|} \right) \right|
\sim
	\left| \log\left(\frac{1}{t} \right) \right|
,
\end{gather*}
and thus condition (d) of Theorem \ref{Theorem:Ingham} holds.

By the definition of the Riemann-Stieltjes integral, we have
\begin{gather*}
f(z) = \int_{0}^\infty e^{-zu}dA(u)
=
\sum_{n\geq 0} b_n e^{-zn}
=
B(e^{-z}).
\end{gather*}
By the assumptions on $B$ in Theorem \ref{Corollary:CorToIngham}, $f(z)$ exists for $\RE(z)>0$
and conditions (ii) of Theorem \ref{Theorem:Ingham} and (i') of Theorem \ref{Theorem:InghamPrime}
are satisfied.
Thus all the hypotheses of Theorem \ref{Theorem:Ingham}
are satisfied. We note that $\psi(x)=\sqrt{\frac{\gamma}{x}}$, and thus we have
\begin{align*}
A(x)
=
\sum_{n<x}b_n
\sim
\frac{ \lambda  \gamma^{\frac{\beta}{2}-\frac14 }\log\left( x \right)^\alpha }
{2^{\alpha+1}\sqrt{\pi} x^{\frac{\beta}{2}+\frac14 }}
e^{2\sqrt{\gamma x}}
.
\end{align*}
A short calculation using the expression on the right-hand side shows that $A(N+1) \sim A(N)$, which implies \eqref{asbsum} on plugging in $x = N+1$ to the left-hand side.
\qedhere
\end{proof}

\section{The Euler-Maclaurin summation formula and the proofs of Theorems \ref{Theorem:EulerMaclaurin1DShifted}, \ref{Prop:EulerMaclaurin1DPoleShifted}, and \ref{Theorem:EulerMaclaurinGeneral}}
\label{S:Euler-Mac}

\subsection{The one-dimensional case}
In this sub-section, we prove Theorems \ref{Theorem:EulerMaclaurin1DShifted} and \ref{Prop:EulerMaclaurin1DPoleShifted}.
We make use of Taylor's theorem in the following form:
Suppose that $f$ is holomorphic in a neighborhood containing $C_R(0)$, the circle of radius $R$ centered at the origin. If $|z|<R$, then
\begin{gather}\label{Taylor}
f(z)
=
	\sum_{k=0}^{N-1} \frac{f^{(k)}(0)}{k!} z^k
	+
	\frac{z^N}{2\pi i}\int_{C_R(0)} \frac{f(w)}{w^N(w-z)}dw
.
\end{gather}
In particular, for $z$ sufficiently small,
\begin{equation}
\label{E:Taylor}
\sum_{k\geq N} \frac{f^{(k)}(0)}{k!}z^k
\ll
	z^N \max_{|w|=R}\left|f(w)\right|
.
\end{equation}

We are now ready to prove Theorem \ref{Theorem:EulerMaclaurin1DShifted}.

\begin{proof}[Proof of Theorem \ref{Theorem:EulerMaclaurin1DShifted}]
The assumption that $f$ has sufficient decay ensures that the sum converges, and also implies that $wf(w) \to 0$ uniformly as $|w| \to \infty$ in $D_\theta$. Finally,
if $n \in \N_0$ is fixed and $|\alpha|\le\theta$, then we have
\begin{align}
\label{E:decayCondns}
\int_0^{(\cos(\alpha)+i\sin(\alpha))\infty} \left|f^{(n)}(z)\right|dz &= O_n(1) \quad \mbox{uniformly in $\alpha$},
\end{align}
where throughout the proof, all integrals are taken along straight lines. The claim in \eqref{E:decayCondns} follows by splitting the integral as
\begin{align*}
\int_0^{\left(\cos(\alpha)+i\sin(\alpha)\right)C_1(n)} \left|f^{(n)}(z)\right| dz
+
\int_{\left(\cos(\alpha)+i\sin(\alpha)\right)C_1(n)}^{\left(\cos(\alpha)+i\sin(\alpha)\right)\infty} \left|f^{(n)}(z)\right| dz,
\end{align*}
where $C_1(n)$ is a constant such that $|w|\ge C_1(n)$ implies that $|f^{(n)}(w)|\le C_2(n)|w|^{-1-\varepsilon_n}$ for some $C_2(n)$. The second piece is then clearly uniformly bounded, and the first piece is as well, due to the fact that the region $|\Arg(w)|\le \theta$ and $|w|\le C_1(n)$ is compact (and $f^{(n)}(w)$ is continuous).

We now present the fundamental identities that underlie Euler-Maclaurin summation, which follow from integration by parts and properties of Bernoulli polynomials. Throughout, supposing that $\mu\in\R$ is fixed, we take $w$ sufficiently small so that $f$ is holomorphic in a region containing the line segment $[w(\mu-1),0]$. We use \eqref{E:B1}, \eqref{E:B'}, and the fact that $B_0(x)=1$ to obtain
\begin{align*}
\int_0^1 f(w(x+\mu-1))  dx
&=
	\frac12 \left( f(w\mu)+ f(w(\mu-1)) \right)
	-w\int_0^1 f'(w(x+\mu-1)) B_{1}(x)dx
.
\end{align*}
Next, for $n \geq 1$, we use \eqref{E:BnDef} and \eqref{E:B'} to conclude that
\begin{align*}
&\hspace{-1.5em}\int_0^1  \frac{f^{(n)}(w(x+\mu-1))  B_n(x)}{n!} dx
\\
&=
	\frac{B_{n+1}}{(n+1)!}\left(f^{(n)}(w\mu) - f^{(n)}(w(\mu-1)) \right)
	-w\int_0^1 \frac{f^{(n+1)}(w(x+\mu-1)) B_{n+1}(x)}{(n+1)!} dx
.
\end{align*}

Using induction on $N\in\N$, one may then show that
\begin{align}
\nonumber
&\int_{0}^1 f(w(x+\mu-1)) dx\\
&=\frac{1}{2}\left(f(w\mu)+f(w(\mu-1)) \right)
+
\sum_{n=1}^{N-1}\frac{(-1)^n B_{n+1}}{(n+1)!}
\left(f^{(n)}(w\mu)-f^{(n)}(w(\mu-1)) \right) w^n
\nonumber\\&\hspace{6cm}\quad
+
(-1)^Nw^N\int_0^1 \frac{f^{(N)}(w(x+\mu-1)) B_N(x)}{N!} dx
.\label{EqEulerMaclaurinProofPart1}
\end{align}

We take $w$ sufficiently small, so that  $f$ is holomorphic in a region containing
the line segment $[wa,0]$.
Summing \eqref{EqEulerMaclaurinProofPart1} with $\mu = m+a$,
for $m\in \{1,2,\dotsc,M\}$, gives that
\begin{align}
\notag
\int_a^{M+a} f(wx) dx
&=
	\frac12\sum_{m=1}^M \left( f(w(m+a)) + f(w(m+a-1)) \right)	
	\\ \notag &\quad
	+
	\sum_{n=1}^{N-1}
	\frac{(-1)^nB_{n+1}}{(n+1)!}
	\sum_{m=1}^M \left(f^{(n)}(w(m+a)) - f^{(n)}(w(m+a-1)) \right)w^n
	\\ \notag &\quad	
	+
	(-1)^Nw^N\sum_{m=1}^M \int_0^1 \frac{ f^{(N)}(w(x+m+a-1)) B_N(x)}{N!} dx
\\ \notag
&=
	\frac12\left(f(wa)+f(w(M+a)) \right)
	+\sum_{m=1}^{M-1} f(w(m+a))
	\\ \notag &\quad	
	+
	\sum_{n=1}^{N-1} \frac{(-1)^nB_{n+1}}{(n+1)!}
	\left( f^{(n)}(w(M+a)) - f^{(n)}(wa) \right)w^n
	\\ \notag &\quad
	+
	(-1)^Nw^N \int_a^{M+a} \frac{ f^{(N)}(wx) \widetilde{B}_N(x-a)}{N!} dx
,
\end{align}

\noindent where the {\it $N$-th periodic Bernoulli polynomial} is defined by $\widetilde{B}_N(x) := B_N(x - \lfloor x\rfloor)$.
Substituting $z=wx$ in the integrals, we obtain
\begin{align*}
\frac{1}{w}\int_{wa}^{w(M+a)} \!\! f(z) dz
&=
	\sum_{m=0}^{M-1} f(w(m+a))
	+
	\sum_{n=0}^{N-1} \frac{B_{n+1}(0) \left( f^{(n)}(wa) - f^{(n)}(w(M+a)) \right) }{(n+1)!}w^n
	\\&\quad
	+
	(-1)^{N}w^{N-1} \int_{wa}^{w(M+a)} \frac{ f^{(N)}(z) \widetilde{B}_N\left(\frac{z}{w}-a\right)}{N!} dz
,
\end{align*}
which in the limit $M\rightarrow\infty$ becomes
\begin{align}\label{Eq:EulerMaclaurinShifted1}
\sum_{m\geq 0} f(w(m+a))
&=
	\frac{1}{w}\int_{wa}^{w\infty} f(z) dz
	-\sum_{n=0}^{N-1} \frac{B_{n+1}(0) f^{(n)}(wa)}{(n+1)!}w^n
	\nonumber\\&\qquad\qquad\qquad\qquad\qquad
	-
	(-1)^{N}w^{N-1} \int_{wa}^{w\infty} \frac{ f^{(N)}(z) \widetilde{B}_N\left(\frac{z}{w}-a\right)}{N!} dz
.
\end{align}

We now claim that
\begin{gather}\label{claimint}
\int_0^{w\infty } f(z) dz
=
\int_0^{\infty} f(x) dx.
\end{gather}
In particular, since $f$ has no poles in $D_\theta$, the Residue theorem implies that for $r > 0$ we have (writing $\alpha:=\Arg(w)$)
\begin{equation*}
\int_0^{r \cos(\alpha)} f(x) dx + \int_{r \cos(\alpha)}^{r(\cos(\alpha) + i \sin(\alpha))} f(z) dz
= \int_0^{wr} f(z) dz.
\end{equation*}
The second integral vanishes as $r \to \infty$, since
\begin{align*}
\int_{r\cos(\alpha)}^{r(\cos(\alpha)+i\sin(\alpha)} f(z)dz
&\ll
r\sin(|\alpha|)	\max_{|c| \le r\sin(|\alpha|) } |f(r\cos(\alpha)+ci)  |
\\
&\le
r\sin(\theta) \max_{|c| \le r\sin(\theta) } |f(r\cos(\alpha)+ci)  |
\quad \rightarrow \quad 0.
\end{align*}
This yields \eqref{claimint}.

Moreover, for $w$ sufficiently small, we have that
\begin{gather*}
\int_0^{wa} f(z) dz
=
\int_0^{wa} \sum_{k\geq0} \frac{f^{(k)}(0)}{k!}z^k dz
=
\sum_{k\geq0} \frac{f^{(k)}(0)a^{k+1}}{(k+1)!}w^{k+1}
.
\end{gather*}
Thus \eqref{Eq:EulerMaclaurinShifted1} becomes
\begin{align}\label{Eq:EulerMaclaurinShifted2}
\sum_{m\geq0} f(w(m+a))
&=
	\frac{1}{w}\int_{0}^{\infty } f(x) dx
	-
	\sum_{k\geq0} \frac{f^{(k)}(0)a^{k+1}}{(k+1)!}w^{k}
	-
	\sum_{n=0}^{N-1} \frac{B_{n+1}(0) f^{(n)}(wa)}{(n+1)!}w^n
	\nonumber\\&\qquad\qquad\qquad\qquad\qquad
	-
	(-1)^{N}w^{N-1} \int_{wa}^{w\infty} \frac{ f^{(N)}(z) \widetilde{B}_N\left(\frac{z}{w}-a\right)}{N!} dz
.
\end{align}
In order to obtain the desired expression, we plug \eqref{Taylor} into the second sum \eqref{Eq:EulerMaclaurinShifted2}, finding that
\begin{align*}
&\sum_{n=0}^{N-1} \frac{B_{n+1}(0) f^{(n)}(wa)}{(n+1)!}w^n \\ \nonumber
& =
	\sum_{n=0}^{N-1} \frac{B_{n+1}(0) w^n}{(n+1)!} \left(\sum_{m = 0}^{N-n-1} \frac{f^{(n+m)}(0) (wa)^m}{m!} + \frac{(wa)^{N-n}}{2 \pi i} \int_{C_R(0)} \frac{f^{(n)}(z)}{z^{N-n} (z-wa)} dz\right)
\nonumber\\
&=
	\sum_{k=0}^{N-1}
	\frac{f^{(k)}(0) w^{k}}{(k+1)!}
	\sum_{n=0}^{k} \binom{k+1}{n+1} B_{n+1}(0) a^{k-n}
	+
	\frac{w^{N}}{2\pi i}
	\sum_{n=0}^{N-1} \frac{B_{n+1}(0)a^{N-n}}{(n+1)!}
	\int_{C_R(0)}
	\frac{f^{(n)}(z)}{z^{N-n}(z-wa)}dz.
\end{align*}
Here $R$ is chosen such that $C_R(0)$ is contained in the domain in which $f$ is holomorphic. The first sum above is then further simplified using \eqref{E:Bkx+y}, as this implies that
\begin{align*}
\sum_{n=0}^{k} \binom{k+1}{n+1} B_{n+1}(0) a^{k-n}
&=
-a^{k+1}
		+ \sum_{n=0}^{k+1} \binom{k+1}{n} B_{n}(0) a^{k+1-n}
= -a^{k+1} + B_{k+1}(a).
\end{align*}
Thus \eqref{Eq:EulerMaclaurinShifted2} becomes
\begin{align}
\label{Eq:EulerMaclaurinShiftedFinal}
\sum_{m\geq0} f(w(m+a))
&=
	\frac{1}{w}\int_{0}^{\infty } f(x) dx
	-
	\sum_{n=0}^{N-1}
	\frac{B_{n+1}(a) f^{(n)}(0)}{(n+1)!}w^{n}
	-
	\sum_{k\geq N} \frac{f^{(k)}(0)a^{k+1}}{(k+1)!}w^{k}
	\nonumber\\&\quad
	-
	\frac{w^{N}}{2\pi i}
	\sum_{n=0}^{N-1} \frac{B_{n+1}(0)a^{N-n}}{(n+1)!}
	\int_{C_R(0)}
	\frac{f^{(n)}(z)}{z^{N-n}(z-wa)}dz
	\nonumber\\&\quad
	-
	(-1)^{N}w^{N-1} \int_{wa}^{w\infty} \frac{ f^{(N)}(z) \widetilde{B}_N\left(\frac{z}{w}-a\right)}{N!} dz
.
\end{align}
We now claim that all terms except the first and the second on the right-hand side are $O(w^N)$. For the third term, we use \eqref{E:Taylor} to find that
\begin{gather*}
\sum_{k\geq N} \frac{f^{(k)}(0)a^{k+1}}{(k+1)!}w^{k}
=
\frac{1}{w}
\int_{0}^{wa}\sum_{k\geq N} \frac{f^{(k)}(0)z^k}{k!}dz
\ll
	w^{N} \max_{|z|=R}|f(z)|
\ll
	w^{N}.
\end{gather*}
For the fourth term, we only need to show that the integral is uniformly bounded as $w \to 0$ in $D_\theta$.
Since $f$ is holomorphic on $C_R(0)$, $f^{(n)}(z)$ is uniformly bounded, and $|\frac{1}{z^{N-n}}| = \frac{1}{R^{N-n}}$.
 Furthermore, $\frac{1}{z-wa}$ is uniformly bounded on $|aw| < \frac{R}{2}$.
 This yields the claim for the fourth term.

Finally, the fact that $\widetilde{B}_N(x)$ is periodic implies that $\widetilde{B}_N(\frac{z}{w}-a)$ is bounded on the ray
from the origin through $w$. We therefore have the following bound on the fifth term,
\begin{gather*}
(-1)^{N}w^{N-1} \int_{wa}^{w\infty} \frac{ f^{(N)}(z) \widetilde{B}_N\left(\frac{z}{w}-a\right)}{N!} dz
\ll	
	w^{N-1} \int_0^{w\infty}  \left|f^{(N)}(z)\right| dz
\ll	
	w^{N-1}
,
\end{gather*}
where the final bound follows from  \eqref{E:decayCondns}. This completes the proof.
\end{proof}

We record one particularly useful corollary to Theorem \ref{Theorem:EulerMaclaurin1DShifted} that allows for alternating signs.
\begin{corollary}\label{Cor:EulerMaclaurinBooleShift}
Suppose that $0\le \theta < \frac{\pi}{2}$.
Let $f:\C\rightarrow\C$ be holomorphic in a domain containing
$D_\theta$, and assume that
$f$ and all of its derivatives are of sufficient decay
as $|w|\rightarrow\infty$ in $D_\theta$.
Then for $a \in \R$ and $N\in\N_0$,
\begin{gather*}
\sum_{m\geq0} (-1)^m f(w(m+a))
=
	\frac{1}{2}\sum_{n=0}^{N-1} \frac{E_{n}(a) f^{(n)}(0)}{n!}w^n
	+
	O_N\left(w^N\right)
,
\end{gather*}
uniformly, as $w\rightarrow0$ in $D_\theta$. Recall that the Euler polynomials are defined in (\ref{Euler}).
\end{corollary}
\begin{proof}
We write
\begin{align*}
\sum_{m\geq0} (-1)^m f(w(m+a))
&=
	\sum_{r\in \{0,1\}} (-1)^r
	\sum_{m\geq0} f\left(2w\!\left(m+\frac{a}{2}+\frac{r}{2}\right)\right)
,
\end{align*}
and thus we may apply Theorem \ref{Theorem:EulerMaclaurin1DShifted} with $w\mapsto 2w$ and $a \mapsto \frac{a}{2} + \frac{r}{2}$. This yields
\begin{align*}
\sum_{m\geq 0} (-1)^m f(w(m+a))
&=
	-\sum_{n=0}^{N-1} \frac{\left( B_{n+1}\left(\frac{a}{2}\right) - B_{n+1}\left(\frac{a}{2}+\frac{1}{2}\right) \right) f^{(n)}(0)}{(n+1)!}(2w)^n
	+
	O_N\left(w^N\right)
\\
&=
	\frac{1}{2}
	\sum_{n=0}^{N-1}\frac{ E_{n}(a) f^{(n)}(0)}{n!}w^n
	+
	O_N\left(w^N\right)
,
\end{align*}
where in the second equality we use \eqref{E:Bn=En}.
\end{proof}

\begin{remark}
The expansion in Corollary \ref{Cor:EulerMaclaurinBooleShift} can alternatively be
proven using the Euler-Boole summation formula (see \cite[equation (5)]{BCM}), namely
\begin{align*}
\sum_{m=0}^{M-1} (-1)^m f(m+a)
&=
	\frac{1}{2}\sum_{n=0}^{N-1}\frac{E_n(a)}{n!}
	\left( f^{(n)}(0) + (-1)^{M-1} f^{(n)}(M) \right)
\\&\quad
	+
	\frac{1}{2(N-1)!}\int_{0}^M
	f^{(N)}(x) \widetilde{E}_{N-1}(a-x) dx
,
\end{align*}
where $0\le a<1$ and
$\widetilde{E}_{n}(x)$ are the periodic Euler functions defined through
the Euler polynomials by $\widetilde{E}_n(x):=E_n(x)$ for $0 \leq x<1$, and
$\widetilde{E}_n(x+1):=-\widetilde{E}_n(x)$.

More generally, one sees that the method used in the proof of Corollary  \ref{Cor:EulerMaclaurinBooleShift} applies to sums of the form $\sum_m \chi(m)f(wm)$,
where $\chi$ is periodic. Of specific interest are the cases when the
average of $\chi$ is zero, such as with non-principal Dirichlet characters, as the integral terms of Theorem
\ref{Theorem:EulerMaclaurin1DShifted} then cancel, leaving an asymptotic expansion that  converges at $w=0$.
\end{remark}

In the case that there is a simple pole at zero, the main new ingredient in the proof is the use of series representations for the complex logarithm digamma function in order to identify the asymptotic contribution of the pole, as we use Theorem \ref{Theorem:EulerMaclaurin1DShifted} to obtain the remainder of the asymptotic expansion.

\begin{proof}[Proof of Theorem \ref{Prop:EulerMaclaurin1DPoleShifted}]
Set
\begin{gather*}
g(w) := \frac{b_{-1}e^{-w}}{w},  \qquad\qquad
h(w) := f(w) - g(w).
\end{gather*}
Using the series expansion of the exponential function, we obtain that
\begin{gather*}
g(w) =  b_{-1} \sum_{n\geq-1} \frac{(-1)^{n+1}w^n}{(n+1)!}.
\end{gather*}
In particular, $g$ has a simple pole at the origin with residue $b_{-1}$,
and thus $h$ is holomorphic at the origin. By Theorem
\ref{Theorem:EulerMaclaurin1DShifted},
we have that
\begin{multline}
\sum_{m\geq0} h(w(m+a))=
	\frac{1}{w}\int_0^\infty \left( f(x) - \frac{b_{-1}e^{-x}}{x} \right) dx	\\
	-
	\sum_{n=0}^{N-1} \frac{B_{n+1}(a) }{n+1}
	\left( b_n + \frac{b_{-1}(-1)^n}{(n+1)!} \right)w^n+
	O_N\left(w^N\right)
.
\notag
\end{multline}

The claim follows once we show that for $\RE(w)>0$,
\begin{gather}\label{Eq:EulerMaclaurinLogShiftedEq3}
-b_{-1} \sum_{n\geq0} \frac{B_{n+1}(a)}{(n+1)(n+1)!}(-w)^{n}
=
b_{-1}\frac{\Log\left(\frac{1}{w}\right)}{w} - \sum_{m\geq0} g(w(m+a))
	+
	\frac{b_{-1}C_a}{w}
.
\end{gather}
Since for $\RE(w)>0$, we have
\begin{align*}
\sum_{m\geq0} g(w(m+a))
&
=
	\frac{b_{-1}}{w}\sum_{m\geq0} \frac{e^{-w(m+a)}}{m+a}
,
\end{align*}
we may instead prove the equivalent identity,
\begin{gather}\label{Eq:EulerMaclaurinLogShiftedEq4}
\Log\left(\frac{1}{w}\right) - \sum_{m\geq0} \frac{e^{-w(m+a)}}{m+a} + C_a
=
\sum_{n\geq0} \frac{B_{n+1}(a)(-w)^{n+1}}{(n+1)(n+1)!}.
\end{gather}
To see \eqref{Eq:EulerMaclaurinLogShiftedEq4}, we first note that
\begin{align*}
\frac{d}{dw} \sum_{n\geq0} \frac{B_{n+1}(a)(-w)^{n+1}}{(n+1)(n+1)!}
&=
	-\sum_{n\geq0} \frac{B_{n+1}(a)(-w)^{n}}{(n+1)!}
=
	-\frac1{w} + \frac1{w}\sum_{n\geq0} \frac{B_{n}(a)(-w)^n}{n!}
\\
&=
	-\frac1{w} - \frac{e^{-wa}}{e^{-w}-1}
=
	-\frac1{w} + \sum_{m\geq0} e^{-(m+a)w}
\\
&=
	\frac{d}{dw}\left(\Log\left(\frac{1}{w}\right)
	-	
	\sum_{m\geq0} \frac{e^{-(m+a)w}}{m+a}\right).
\end{align*}
Thus we have
\begin{align*}
\sum_{n\geq0} \frac{B_{n+1}(a)(-w)^{n+1}}{(n+1)(n+1)!}
&=
\Log\left(\frac{1}{w}\right) - \sum_{m\geq0} \frac{e^{-(m+a)w}}{m+a} + C
,
\end{align*}
for some constant $C$, and we see that the left hand-side provides an analytic
continuation of the right hand-side in a neighborhood of $w=0$. However, the
left hand-side is clearly zero when $w=0$. To evaluate the limit of
the right hand-side, as $w\rightarrow0$ with $\RE(w)>0$, we first
note that
\begin{align}\label{E:Logsums}
&
	\Log\left(\frac{1}{w}\right) - \sum_{m \geq 0} \frac{e^{-w(m+a)}}{m+a}
\notag
	=\Log\left(\frac{1}{w}\right) - e^{-wa}\sum_{m \geq 0} \frac{e^{-mw}}{m+1}
	- e^{-aw} \sum_{m \geq 0} e^{-mw} \left(\frac{1}{m+a} - \frac{1}{m+1}\right)
\notag\\
&=
	\Log\left(\frac{1}{w}\right) + e^{(1-a)w} \Log\left(1 - e^{-w}\right)
	- e^{-aw} (1-a) \sum_{m \geq 0} \frac{e^{-mw}}{(m+a)(m+1)}
\notag\\
&=
	\Log\left(\frac{1}{w}\right)\left(1 - e^{(1-a)w}\right)  + e^{(1-a)w} \Log\left(\frac{1 - e^{-w}}{w}\right)
	- e^{-aw} (1-a) \sum_{m \geq 0} \frac{ e^{-mw}}{(m+a)(m+1)}
,
\end{align}
where in the final equality we use that
\begin{gather*}
\Log\left(1 - e^{-w}\right)
=
	\Log\left(\frac{1 - e^{-w}}{w}\right)
	-\Log\left(\frac{1}{w}\right),
\end{gather*}
for $\RE(w)>0$, since both $1-e^{-w}$ and $\frac1w$ lie in the right half-plane. A short calculation with l'Hospital's rule shows that the first two terms
in \eqref{E:Logsums} tend to zero (as $w \to 0^+$), and since the convergence of the series is uniform in $w$, for $\RE(w)>0$, we have
\begin{gather*}
\lim_{w\rightarrow 0}
\left(
	\Log\left(\frac{1}{w}\right) - \sum_{m \geq 0} \frac{e^{-w(m+a)}}{m+a}
\right)
=
	- (1-a)\sum_{m \geq 0} \frac{1}{(m+a)(m+1)}
=
	-C_a.
\end{gather*}
Therefore \eqref{Eq:EulerMaclaurinLogShiftedEq4}, and as such \eqref{Eq:EulerMaclaurinLogShiftedEq3},
holds.
\end{proof}

\subsection{The multivariable Euler-Maclaurin summation formula}
We now turn to the multivariable version  of
the Euler-Maclaurin asymptotic expansion stated in Theorem \ref{Theorem:EulerMaclaurinGeneral}.
The following proposition is a refined version that enables a proof by induction.

\begin{proposition}\label{PropositionEulerMaclaurinWithError}
Suppose that $0 \leq \theta_j < \frac{\pi}{2}$ for $1 \leq j \leq s$, and that $f:\C^s\rightarrow\C$ is holomorphic in a domain containing $D_{\b \theta}$.
If $f$ and all of its derivatives are of sufficient decay in $D_{\b \theta}$, then
 for $1 \leq r < s$, $\b{a}\in\R^s$ and $N\in\N_0$, we have
\begin{align}\label{Eq:EMFull}
&\frac{1}{w^r}\int_{[0,\infty)^r} f(\b{x}) dx_1\dotsb dx_r
\nonumber\\
&=
\sum_{\b{m}\in \N_0^r } f( w(\b{m}+\b{a}), \b{x}_{r+1,s} )
	-
	(-1)^r\sum_{ \b{n}\in \NN^r}	
	f^{(\b{n}, \b{0}_{s-r} )}( \b{0}_r, \b{x}_{r+1,s} )	
	\prod_{1\le j\le r} \frac{B_{n_j+1}(a_j)}{(n_j+1)!}w^{n_j} 	
	\nonumber\\&\quad
	-
	\sum_{\emptyset\subsetneq \mathscr{S} \subsetneq \{1,\dots,r\}}	 
	\frac{(-1)^{|\mathscr{S}|}}{w^{r-|\mathscr{S}|}}
	\sum_{\substack{ n_j\in \mathcal N_N\\ j\in\mathscr{S}}}
	\int_{[0,\infty)^{r-|\mathscr{S}|}}
	\left[
		\prod_{j\in\mathscr{S}}
		\frac{\partial^{n_j}}{\partial x_j^{n_j}}
		f(\b{x})
	\right]_{\substack{x_j=0\\ j\in\mathscr{S}}}
	\prod_{\substack{1\le k\le r\\[0.5ex] k\not\in\mathscr{S}}} dx_k
	\prod_{j\in\mathscr{S}} \frac{B_{n_j+1}(a_j)}{(n_j+1)!}w^{n_j}
	\nonumber\\&\quad
	+
	w^{N-r}g_{r,w}(\b{x}_{r+1,s})
,
\end{align}
where
$\b{x}\in\R^s$,
$\b{x}_{r+1,s} := (x_{r+1},\dots, x_s)$,
$\b{0}_j$ is the zero vector of length $j$, and
$g_{r,w}:\C^{s-r}\rightarrow\C$ is such that
$g_{r,w}(\b{x}_{r+1,s})\ll 1$ uniformly in $\b{x}_{r+1,s}$ as
$w\rightarrow 0$ in $D_{\theta_1}\!\cap\dotsb\cap D_{\theta_r}$ and
\begin{align*}
\int_{[0,\infty)^j} |g_{r,w}(\b{x}_{r+1,s} )| dx_{r+1}\dotsm dx_{r+j}
\ll 1,
\end{align*}
uniformly in $w$ and $(x_{r+j+1},\dotsc,x_s)$
for $1\le j\le s-r$. When $r=s$, \eqref{Eq:EMFull}
holds with $\b{x}_{r+1,s}$ being the empty vector and $g_{s,w}$ a function of $w$
such that $g_{s,w}\ll 1$ as $w\rightarrow 0$ in
$D_{\theta_1}\!\cap\dotsb\cap D_{\theta_r}$.
\end{proposition}
\begin{proof}
Throughout the proof $s \in \N$ is fixed, and we proceed by induction on $r$.
At various points in the proof, we consider $f^{(n_1, \dots, n_s)}$, with each $n_k \leq N$,
and restrict this derivative to a function of the single variable $x_j$
(with all other variables held fixed).
We choose $R > 0$ such that all such functions are holomorphic in a neighborhood containing $C_R(0)$.
This is possible because the decay assumption implies that $R$ can be chosen uniformly for each
individual function, and there are finitely many in total.

The base case of $r=1$
is \eqref{Eq:EulerMaclaurinShiftedFinal}
with $f(x) \mapsto f(x,\b{x}_{2,s})$ and the resulting
$g_{1,w}(\b{x}_{2,s})$ is
\begin{align*}
g_{1,w}(\b{x}_{2,s})
&:=
	w^{1-N}\sum_{k_1\geq N} \frac{f^{(k_1,\b{0}_{s-1})} (0,\b{x}_{2,s}) a_1^{k_1+1}}{(k_1+1)!}w^{k_1}
	\\&\quad
	+
	\frac{w}{2\pi i}\sum_{n_1\in\NN}
	\frac{ B_{n_1+1}(0)a_1^{N-n_1} }{(n_1+1)!}
	\int_{C_R(0)}
	\frac{ f^{(n_1,\b{0}_{s-1})}(z_1,\b{x}_{2,s}) }{z_1^{N-n_1}(z_1-wa_1)}
	dz_1
	\\&\quad	
	+	
	(-1)^{N}w\int_{a_1}^{\infty} \frac{f^{(N,\b{0}_{s-1})}(wx_1,\b{x}_{2,s}) \widetilde{B}_N(x_1-a_1)}{N!}dx_1
,
\end{align*}
where $R>0$ is sufficiently small. Due to the assumption of sufficient decay,
the integrals converge uniformly
in $\b{x}_{2,s}$ and thus we see that the second and third terms above meet the conditions required of $g_{1,w}$.
For the first term, \eqref{E:Taylor} gives
\begin{gather*}
w^{1-N}\sum_{k_1\geq N} \frac{f^{(k_1,\b{0}_{s-1})} (0,\b{x}_{2,s}) a_1^{k_1+1}}{(k_1+1)!}w^{k_1}
\ll
	w\max_{|z_1|=R} |f(z_1,\b{x}_{2,s})|
,
\end{gather*}
which also meets the conditions required of $g_{1,w}$ due to the assumption of sufficient decay.

We now fix $r$ with $2\le r\le s$, assume that  \eqref{Eq:EMFull} is true for $r-1$,
and verify that it holds for $r$. We let
$\b{x}_{r-1}:=(x_1,\dotsc,x_{r-1})$
and
$\b{a}_{r-1}:=(a_1,\dots,a_{r-1})$.
By assumption, we can take \eqref{Eq:EMFull} with $r-1$ and integrate with respect to
$x_{r}$, and also divide by $w$, which yields
\begin{align}\label{Eq:EulerMaclaurinMultidimensional1}
&\frac{1}{w^{r}}\int_{[0,\infty)^{r}} f(\b{x}) dx_1\dotsb dx_{r}
\nonumber\\
&=
	\frac{1}{w} \sum_{\b{m}\in \N_0^{r-1} }
		\int_{0}^\infty f( w(\b{m}+\b{a}_{r-1}), x_{r}, \b{x}_{r+1,s} ) dx_r
	\nonumber\\&\quad
	+
	\frac{(-1)^r}{w}
	\sum_{ \b{n}\in \mathcal N_N^{r-1}}	
	\int_0^\infty f^{(\b{n}, \b{0}_{s-r+1} )}( \b{0}_{r-1}, x_{r},\b{x}_{r+1,s} ) dx_{r}	
	\prod_{1\le j \leq r-1} \frac{B_{n_j+1}(a_j)}{(n_j+1)!}w^{n_j}
	\nonumber\\&\quad
	-
	\!\!
	\sum_{\emptyset\subsetneq \mathscr{S} \subsetneq \{1,\dots,r-1\}}	
	\!\!
	\frac{(-1)^{|\mathscr{S}|}}{w^{r-|\mathscr{S}|}}
	\sum_{\substack{ n_j\in\mathcal N_N\\ j\in\mathscr{S}}}
	\int_{[0,\infty)^{r-|\mathscr{S}|}}
	\left[
		\prod_{j\in\mathscr{S}}
		\frac{\partial^{n_j}}{\partial x_j^{n_j}}
		f(\b{x})
	\right]_{\substack{x_j=0\\ j\in\mathscr{S}}}
	\prod_{\substack{1\le k\le r\\[0.5ex] k\not\in\mathscr{S}}} dx_k
	\prod_{j\in\mathscr{S}} \frac{B_{n_j+1}(a_j)}{(n_j+1)!}w^{n_j}
	\nonumber\\&\quad
	+
	w^{N-r}h_{1,w}(\b{x}_{r+1,s})
,
\end{align}
where
\begin{align*}
h_{1,w}(\b{x}_{r+1,s})
:=
	\int_0^\infty g_{r-1,w}(x_{r},\b{x}_{r+1,s}) dx_{r}
.
\end{align*}
The restrictions on $g_{r-1,w}$ give that
$h_{1,w}(\b{x}_{r+1,s})$ satisfies the conditions required of $g_{r,w}$.

However, for fixed $\b{m}\in\N_0^{r-1}$, by \eqref{Eq:EulerMaclaurinShiftedFinal} we have
\begin{align*}
&\frac{1}{w}	\int_{0}^\infty f( w(\b{m}+\b{a}_{r-1}), x_r, \b{x}_{r+1,s} ) dx_r
\\
&\qquad\quad=
	\sum_{m_{r}\geq0} f(w(\b{m}+\b{a}_{r-1}),w(m_{r}+a_{r}), \b{x}_{r+1,s})
	\nonumber\\&\qquad\qquad
	+
	\sum_{n_r \in \mathcal N _N} \frac{B_{n_r+1}(a_r)f^{(\b{0}_{r-1},n_r,\b{0}_{s-r})}(w(\b{m}+\b{a}_{r-1}),0,\b{x}_{r+1,s})}{(n_r+1)!}w^{n_r}
	\nonumber\\&\qquad\qquad
	+
	\sum_{k_r\geq N} \frac{f^{(\b{0}_{r-1},k_r,\b{0}_{s-r})} (w(\b{m}+\b{a}_{r-1}),0,\b{x}_{r+1,s}) a_r^{k_r+1}}{(k_r+1)!}w^{k_r}
	\\&\qquad\qquad
	+
	\frac{w^{N}}{2\pi i}\sum_{n_r\in\NN}
	\frac{ B_{n_r+1}(0)a_r^{N-n_r} }{(n_r+1)!}
	\int_{C_R(0)}
	\frac{ f^{(\b{0}_{r-1},n_r,\b{0}_{s-r})}(w(\b{m}+\b{a}_{r-1}),z_r,\b{x}_{r+1,s}) }{z_r^{N-n_r}(z_r-wa_r)}
	dz_r	
	\nonumber\\&\qquad\qquad
	+
	(-1)^{N}w^{N}	\int_{a_r}^{\infty}
	\frac{  f^{(\b{0}_{r-1},N,\b{0}_{s-r})}(w(\b{m}+\b{a}_{r-1}),wx_r,\b{x}_{r+1,s})  \widetilde{B}_N(x_r-a_r)}{N!}dx_r
.
\end{align*}
This yields that
\begin{align}\label{Eq:EulerMaclaurinMultidimensional2}
&\frac{1}{w} \sum_{\b{m}\in \N_0^{r-1} }
\int_{0}^\infty f( w(\b{m}+\b{a}_{r-1}), x_{r}, \b{x}_{r+1,s} ) dx_r
\nonumber\\
& \qquad =
\sum_{\b{m}\in\N_0^r} f(w(\b{m}+\b{a}),\b{x}_{r+1,s})
+
w^{N-r}h_{2,w}(\b{x}_{r+1,s})  \nonumber\\
&\hspace{8.5em}+
\sum_{n_r\in\mathcal N_N}
\frac{B_{n_r+1}(a_r)}{(n_r+1)!}w^{n_r}	
\sum_{\b{m}\in\N_0^{r-1}}
f^{(\b{0}_{r-1},n_r,\b{0}_{s-r})}(w(\b{m}+\b{a}_{r-1}),0,\b{x}_{r+1,s})	 
,
\end{align}
where	
\begin{align*}
&h_{2,w}(\b{x}_{r+1,s})	:=	
	w^{r-N}	\sum_{\b{m}\in\N_0^{r-1}}		
	\sum_{k_r\geq N} \frac{f^{(\b{0}_{r-1},k_r,\b{0}_{s-r})} (w(\b{m}+\b{a}_{r-1}),0,\b{x}_{r+1,s}) a_r^{k_r+1}}{(k_r+1)!}w^{k_r}
	\\&\quad
	+
	\frac{w^{r}}{2\pi i}
	\sum_{\b{m}\in\N_0^{r-1}}	
	\sum_{n_r\in\NN}
	\frac{ B_{n_r+1}(0)a_r^{N-n_r} }{(n_r+1)!}
	\int_{C_R(0)}
	\frac{ f^{(\b{0}_{r-1},n_r,\b{0}_{s-r})}(w(\b{m}+\b{a}_{r-1}),z_r,\b{x}_{r+1,s}) }{z_r^{N-n_r}(z_r-wa_r)}
	dz_r	
	\\&\quad
	+
	(-1)^{N}w^{r}	
	\sum_{\b{m}\in\N_0^{r-1}}		
	\int_{a_r}^{\infty}
	\frac{  f^{(\b{0}_{r-1},N,\b{0}_{s-r})}(w(\b{m}+\b{a}_{r-1}),wx_r,\b{x}_{r+1,s})  \widetilde{B}_N(x_r-a_r)}{N!}dx_r
.
\end{align*}
We find that $h_{2,w}$ satisfies the conditions of $g_{r,w}$ by reasoning similar to that used for $h_{1,w}$.

For fixed $n_r$, applying \eqref{Eq:EMFull} with $r-1$ gives
\begin{align*}
&\sum_{\b{m}\in\N_0^{r-1}}
f^{(\b{0}_{r-1},n_r,\b{0}_{s-r})}(w(\b{m}+\b{a}_{r-1}),0,\b{x}_{r+1,s})
\\
&=
	\frac{1}{w^{r-1}}\int_{[0,\infty)^{r-1}} f^{(\b{0}_{r-1},n_r,\b{0}_{s-r})}(\b{x}_{r-1},0,\b{x}_{r+1,s}) dx_1\dotsb dx_{r-1}
	\\&\quad
	-
	(-1)^r\sum_{ \b{n}\in \mathcal N_N^{r-1}}	
	f^{(\b{n}, n_r, \b{0}_{s-r} )}( \b{0}_{r-1},0, \b{x}_{r+1,s} )	 
	\prod_{1\le j< r} \frac{B_{n_j+1}(a_j)}{(n_j+1)!}w^{n_j}
	\\&\quad
	+
	\sum_{\emptyset\subsetneq \mathscr{S} \subsetneq \{1,\dots,r-1\}}	
	\frac{(-1)^{|\mathscr{S}|}}{w^{r-1-|\mathscr{S}|}}
	\sum_{\substack{ n_j\in\mathcal N_N\\ j\in\mathscr{S}}}
	\int_{[0,\infty)^{r-1-|\mathscr{S}|}}
	\left[
		\prod_{j\in\mathscr{S}}
		\frac{\partial^{n_j}}{\partial x_j^{n_j}}
		f^{(\b{0}_{r-1},n_r,\b{0}_{s-r})}(\b{x})
	\right]_{\substack{x_r=0\\ x_j=0\\ j\in\mathscr{S}}}
	\prod_{\substack{1\le k< r\\[0.5ex] k\not\in\mathscr{S}}} dx_k
	\nonumber\\&\hspace{19em}\times
	\prod_{j\in\mathscr{S}} \frac{B_{n_j+1}(a_j)}{(n_j+1)!}w^{n_j}
	+
	w^{N-r+1}g_{n_r,r-1,w}(\b{x}_{r+1,s})
,
\end{align*}
where each  $g_{n_r,r-1,w}(\b{x}_{r+1,s})$ satisfies the conditions of $g_{r,w}$.
Plugging this back into \eqref{Eq:EulerMaclaurinMultidimensional2} yields
\begin{align}\label{Eq:EulerMaclaurinMultidimensional3}
&\frac{1}{w}\sum_{\b{m}\in \N_0^{r-1} }
	\int_{0}^\infty f( w(\b{m}+\b{a}_{r-1}), x_{r}, \b{x}_{r+1,s} ) dx_r
\nonumber\\
&=
	\sum_{\b{m}\in\N_0^r} f(w(\b{m}+\b{a}),\b{x}_{r+1,s})
	\nonumber\\&\quad
	+
	\frac{1}{w^{r-1}}	
	\sum_{n_r\in \NN}
	\frac{B_{n_r+1}(a_r)}{(n_r+1)!}w^{n_r}	
	\int_{[0,\infty)^{r-1}} f^{(\b{0}_{r-1},n_r,\b{0}_{s-r})}(\b{x}_{r-1},0,\b{x}_{r+1,s}) dx_1\dotsb dx_{r-1}
	\nonumber\\&\quad
	-
	(-1)^r	
	\sum_{ \b{n}\in \mathcal N_N^{r}}	
	f^{(\b{n}, \b{0}_{s-r} )}( \b{0}, \b{x}_{r+1,s} )	
	\prod_{1\le j\le r} \frac{B_{n_j+1}(a_j)}{(n_j+1)!}w^{n_j}
	\nonumber\\&\quad
	+
	\sum_{\emptyset\subsetneq \mathscr{S} \subsetneq \{1,\dots,r-1\}}	
	\frac{(-1)^{|\mathscr{S}|}}{w^{r-1-|\mathscr{S}|}}
	\sum_{\substack{ n_j\in \mathcal N_N\\ j\in\mathscr{S}\cup\{r\} }}
	\int_{[0,\infty)^{r-1-|\mathscr{S}|}}
	\left[
		\prod_{j\in\mathscr{S}\cup\{r\}}
		\frac{\partial^{n_j}}{\partial x_j^{n_j}}
		f(\b{x})
	\right]_{\substack{x_j=0\\ j\in\mathscr{S}\cup\{r\} }}
	\prod_{\substack{1\le k<r\\[0.5ex] k\not\in\mathscr{S}}} dx_k
	\nonumber\\&\hspace{16.5em}\times
	\prod_{j\in\mathscr{S}\cup\{r\}} \frac{B_{n_j+1}(a_j)}{(n_j+1)!}w^{n_j}
	+
	w^{N-r}h_{3,w}(\b{x}_{r+1,s})
,
\end{align}
where $h_{3,w}(\b{x}_{r+1,s})$ satisfies the conditions required of $g_{r,w}$,
 since it is the sum of $h_{2,w}$ and the finitely many $g_{n_r,r-1,w}$.

We insert \eqref{Eq:EulerMaclaurinMultidimensional3} back into
\eqref{Eq:EulerMaclaurinMultidimensional1} to find that
\begin{align*}
&\frac{1}{w^{r}}\int_{[0,\infty)^{r}} f(\b{x}) dx_1\dotsb dx_{r}
\nonumber\\
&=
	\sum_{\b{m}\in\N_0^r} f(w(\b{m}+\b{a}),\b{x}_{r+1,s})
	-
	(-1)^r	
	\sum_{ \b{n}\in \mathcal N_N^{r}}	
	f^{(\b{n}, \b{0}_{s-r} )}( \b{0}_r, \b{x}_{r+1,s} )	
	\prod_{1\le j\le r} \frac{B_{n_j+1}(a_j)}{(n_j+1)!}w^{n_j}
	\nonumber\\&\quad
	+
	\frac{1}{w^{r-1}}	
	\sum_{n_r\in\mathcal N_N}
	\frac{B_{n_r+1}(a_r)}{(n_r+1)!}w^{n_r}	
	\int_{[0,\infty)^{r-1}} f^{(\b{0}_{r-1},n_r,\b{0}_{s-r})}(\b{x}_{r-1},0,\b{x}_{r+1,s}) dx_1\dotsb dx_{r-1}
	\nonumber\\&\quad
	+
	\sum_{\emptyset\subsetneq \mathscr{S} \subsetneq \{1,\dots,r-1\}}	
	\frac{(-1)^{|\mathscr{S}|}}{w^{r-1-|\mathscr{S}|}}
	\sum_{\substack{ n_j\in \mathcal N_N\\ j\in\mathscr{S}\cup\{r\} }}
	\int_{[0,\infty)^{r-1-|\mathscr{S}|}}
	\left[
		\prod_{j\in\mathscr{S}\cup\{r\}}
		\frac{\partial^{n_j}}{\partial x_j^{n_j}}
		f(\b{x})
	\right]_{\substack{x_j=0\\ j\in\mathscr{S}\cup\{r\} }}
	\prod_{\substack{1\le k<r\\[0.5ex] k\not\in\mathscr{S}}} dx_k
	\nonumber\\&\hspace{26.5em}\times
	\prod_{j\in\mathscr{S}\cup\{r\}} \frac{B_{n_j+1}(a_j)}{(n_j+1)!}w^{n_j}
	\nonumber\\&\quad
	+
	\frac{(-1)^r}{w}
	\sum_{ \b{n}\in \mathcal N_N^{r-1}}	
	\int_0^\infty f^{(\b{n}, \b{0}_{s-r+1} )}( \b{0}_{r-1}, x_{r},\b{x}_{r+1,s} ) dx_{r}	
	\prod_{1\le j< r} \frac{B_{n_j+1}(a_j)}{(n_j+1)!}w^{n_j}
	\nonumber\\&\quad
	-
	\!	
	\sum_{\emptyset\subsetneq \mathscr{S} \subsetneq \{1,\dots,r-1\}}	
	\!	
	\frac{(-1)^{|\mathscr{S}|}}{w^{r-|\mathscr{S}|}}
	\sum_{\substack{ n_j\in \mathcal N_N\\ j\in\mathscr{S}}}
	\int_{[0,\infty)^{r-|\mathscr{S}|}}
	\left[
		\prod_{j\in\mathscr{S}}
		\frac{\partial^{n_j}}{\partial x_j^{n_j}}
		f(\b{x})
	\right]_{\substack{x_j=0\\ j\in\mathscr{S}}}
	\prod_{\substack{1\le k\le r\\[0.5ex] k\not\in\mathscr{S}}} dx_k
	\prod_{j\in\mathscr{S}} \frac{B_{n_j+1}(a_j)}{(n_j+1)!}w^{n_j}
	\nonumber\\&\quad
	+
	w^{N-r}h_{4,w}(\b{x}_{r+1,s})
,
\end{align*}
where $g_{r,w}(\b{x}_{r+1,s}) := h_{1,w}(\b{x}_{r+1,s}) + h_{3,w}(\b{x}_{r+1,s})$.
Upon inspection, we find that the third, fourth, fifth, and
sixth terms in the right hand-side combine exactly as stated
in \eqref{Eq:EMFull}, so that the proof is complete.
\end{proof}

\section{Concluding Remarks}
\label{S:Conclusion}

There is another variant of the Circle Method due to Wright that is closely related to Ingham's Tauberian theorem. Recall that in Theorem \ref{Corollary:CorToIngham}, the analytic behavior of $B(q)$ in a small region near $q = 1$ is sufficient to determine the asymptotic main terms of the coefficients $b_n$. In particular, for a small, fixed $t > 0$, the conditions in \eqref{additional} require an asymptotic formula for $B(e^{-t})$, and uniform bounds for $B(q)$ along a small arc of radius $e^{-t}$.

In contrast, Wright's Circle Method requires an asymptotic formula for $B(q)$ near $q = 1$ (the ``Major arc''), as well as bounds along the remainder of the circle of radius $e^{-t}$ (the ``Minor arc''). However, the conclusion is also stronger, as one obtains an asymptotic expansion for the coefficients $b_n$, so long as one has an asymptotic expansion for $B(e^{-t})$. Wright first introduced this approach in \cite{Wright33}, and applied it to another example in \cite{Wright71}; in the latter case, he also used Euler-Maclaurin summation to derive the asymptotic expansion for $B$.

In their comprehensive article \cite{NgoR}, Ngo and Rhoades gave a generalized version of Wright's Circle Method. Specifically, Proposition 1.8 in \cite{NgoR} requires that
\begin{equation*}
B\left(e^{-z}\right) = z^\beta e^{\frac{\gamma}{z}} \left(\sum_{s=0}^{N-1} \alpha_s z^s + O \left(z^N\right)\right)
\end{equation*}
in the restricted angle $y \leq \Delta |x|$, as well as
\begin{equation*}
B\left(e^{-z}\right) \ll B\left(e^{-x}\right) e^{-\frac{d}{x}}
\end{equation*}
for some $d > 0$ in the remainder of the circle $|z| = e^{-x}$. In that case, the resulting asymptotic expansion for the coefficients is
\begin{equation}
\label{E:NgoRbn}
b_n = \frac{e^{2 \gamma \sqrt{n}}}{2 \sqrt{\pi} n^{\frac{\beta}{2} + \frac34}} \left(\sum_{s = 0}^{N-1} \left(\sum_{r = 0}^s \alpha_{r} \beta_{s,r-s}\right) n^{-\frac{s}{2}} + O\left(n^{-\frac{N}{2}}\right)\right),
\end{equation}
where $\beta_{s,r}$ are certain combinatorial coefficients. Furthermore, they showed that this result applies to a wide class of functions following essentially the same arguments we discussed in Section \ref{S:Ingham}. In particular, they proved that \eqref{E:NgoRbn} holds if $B(q) = L(q) \xi(q)$, where $\xi(q)$ essentially behaves like a modular form, and $L(q)$ has an asymptotic expansion that is derived using Euler-Maclaurin summation.

\end{document}